%% file: report.tex
\documentclass[reqno]{amsart}

\usepackage[english]{babel}
\usepackage{graphicx}
\usepackage[hidelinks]{hyperref} 
\usepackage{amssymb}
\usepackage{amsmath}
\usepackage[foot]{amsaddr}
\usepackage{amsfonts}
\usepackage{xcolor}
\usepackage{mathtools}
\usepackage{algpseudocode}
\usepackage{algorithm}
\usepackage{comment}
\usepackage{subcaption}
\usepackage{caption}
\usepackage[top=1in, bottom=1.25in, left=1.25in, right=1.25in]{geometry}
\usepackage{float}
\usepackage{tikz}

\newtheorem{example}{Example}{}
{}
{}
\newtheorem{remark}[example]{Remark}

{}
\newtheorem{theorem}{Theorem}

\newtheorem{lemma}{Lemma}

\usepackage{graphicx}
\usepackage{xcolor}

\usepackage{subcaption}
\usepackage{tikz}
\graphicspath{{Images/}}

\input{Pack.sty}

\title[Certified ROMs for space-time parabolic optimal control problems]{A Certified Reduced Basis Method for Linear Parametrized Parabolic Optimal Control Problems in Space-Time Formulation}

\author{Maria Strazzullo$^{1,2}$, Francesco Ballarin$^{3,*}$, Gianluigi Rozza$^{4,*}$}

\address{$^1$ FAST Computing srl,
Via Mazzini 20, 34121, Trieste, Italy}
\address{$^2$ Simplifica srl,
Riva Tommaso Gulli 12
34123, Trieste, Italy}
\address{$^3$ Department of Mathematics for Economic, Financial and Actuarial Sciences, Università Cattolica del Sacro Cuore, via Necchi, 20123, Milan, Italy}
\address{$^4$ mathLab, Mathematics Area, SISSA, via Bonomea 265, I-34136 Trieste, Italy}
\address{$^*$ INdAM-GNCS group member}

\begin{document}
\begin{abstract}
In this work, we propose to efficiently solve time dependent parametrized optimal control problems governed by parabolic partial differential equations through the certified reduced basis method. In particular, we will exploit an error estimator procedure, based on easy-to-compute quantities which guarantees a rigorous and efficient bound for the error of the involved variables. First, we propose analyzing the problem at hand, proving its well-posedness thanks to Ne\v cas - Babu\v ska theory in a space-time formulation. Then, we derive error estimators to apply a Greedy method during the offline stage, to perform, during the online stage, a Galerkin projection onto a low-dimensional
space spanned by properly chosen high-fidelity solutions. We tested the error estimators on two model problems governed by a Graetz flow with {physical and geometrical parametrization}. The results have been compared to a previously proposed bound, based on the exact computation of the Babu\v ska inf-sup constant, in terms of reliability and computational costs. We remark that our findings still hold in the steady setting, and we also propose a brief insight for this simpler formulation.

\end{abstract}

\maketitle
\input{intro}
\input{ocp_formulation}

\input{FE_OCP}
\input{RB_OCP}

\input{Results}

\input{conclusions}

\section*{acknowledgement}
  The computations in this work have been performed with RBniCS \cite{rbnics} library, which is an implementation in FEniCS \cite{fenics} of several reduced order modelling techniques; we acknowledge developers and contributors to both libraries. Computational resources were partially provided by HPC@DISMA, an academic computing center that provides computational resources for academic purposes at DISMA, Politecnico di Torino.
We acknowledge the support by European Union Funding for Research and Innovation -- Horizon 2020 Program -- in the framework of European Research Council Executive Agency: Consolidator Grant H2020 ERC CoG 2015 AROMA-CFD project 681447 ``Advanced Reduced Order Methods with Applications in Computational Fluid Dynamics''. We also acknowledge the PRIN 2017  ``Numerical Analysis for Full and Reduced Order Methods for the efficient and accurate solution of complex systems governed by Partial Differential Equations'' (NA-FROM-PDEs) and the European Union's Horizon 2020 research and innovation program under the Marie Skłodowska-Curie Actions, grant agreement 872442 (ARIA). \\
MS thanks the INdAM - GNCS Project ``Metodi numerici efficienti per problemi accoppiati in sistemi complessi” (CUP E53C24001950001). MS and GR thank the ``20227K44ME - Full and Reduced order modelling of coupled systems: focus on non-matching methods and automatic learning (FaReX)" project – funded by European Union – Next Generation EU  within the PRIN 2022 program (D.D. 104 - 02/02/2022 Ministero dell’Università e della Ricerca). This manuscript reflects only the authors’ views and opinions and the Ministry cannot be considered responsible for them.\\
FB also thanks the project ``Reduced order modelling for numerical simulation of partial differential equations'' funded by Università Cattolica del Sacro Cuore.

\bibliographystyle{abbrvurl}
\bibliography{BIB}
\end{document}

%% file: intro.tex
\section{Introduction}
\label{intro}

In several applications, optimal control problems (OCP($\boldsymbol{\mu}$)s) governed by parametrized partial differential equations (PDE($\boldsymbol{\mu}$)s) can be a versatile tool to better model physical phenomena. A parameter $\boldsymbol{\mu} \in \mathcal P \subset \mathbb R^d$ can represent physical or geometrical configurations and \ocp s respond to the need of parametric studies of controlled systems, where the underlying PDE($\boldsymbol{\mu}$)s is steered to a desired state in order to achieve a specific goal {(see, e.g., \cite{hinze2008optimization} for a thorough introduction to PDE constrained optimization)}. Even though on one hand optimal control is a great {modeling} tool, on the other \ocp s are challenging, not only to {analyze} theoretically but also to deal with in a numerical setting.
Indeed, even if they have been exploited in several research fields, from shape optimization, see e.g.\ \cite{delfour2011shapes,makinen,mohammadi2010applied}, to fluid dynamics, see e.g.\ \cite{dede2007optimal,negri2015reduced,optimal,de2007optimal}, from heamodynamics {\cite{Ballarin2017,LassilaManzoniQuarteroniRozza2013a,ZakiaMaria}} to environmental applications \cite{CARERE2021261,quarteroni2007reduced,Strazzullo1,Strazzullo3,ZakiaMaria},
classical discretization techniques may result in unbearable simulations, which can limit their applicability in many query contexts, where several parametric instances must be studied, possibly in a small amount of time. Furthermore, the computational complexity drastically grows when the governing equation involves time evolution.\\
Due to its great potential in terms of mathematical modelling, time optimization arises in many applications and it has been studied for several PDE($\bmu$)s, see e.g. \cite{HinzeStokes,Iapichino2,leugering2014trends,seymen2014distributed,Stoll1,Stoll}. The goal of this work is to propose a reduced basis (RB) approach to deal with the study of several values of $\bmu \in \mathcal P$ in a low-dimensional framework reducing the computational costs of multiple simulations \cite{hesthaven2015certified,prud2002reliable,RozzaHuynhManzoni2013,RozzaHuynhPatera2008}.
Concerning the employment of reduced techniques to \ocp s, the interested reader may refer to several papers as \cite{bader2016certified,bader2015certified,StrazzulloZuazua,dede2010reduced,gerner2012certified,Iapichino1,karcher2014certified,karcher2018certified,Karcherrev,kunisch2008proper,negri2015reduced,negri2013reduced,PichiStrazzullo,qian2017certified,quarteroni2007reduced,Strazzullo2,STRAZZULLO2023128191,Zoccolan1,zoccolan2}, where the reduction has been implemented through different methodology for a wide range of governing equations.\\
In particular, we seek to extend the approach presented in \cite{negri2013reduced} for steady \ocp s to quadratic optimization models constrained to linear time dependent PDE($\bmu$)s in a space-time framework. {Under the viewpoint of standard discretization techniques, the similarities between steady problems and space-time formulation for optimal control are well known, see for example \cite{r5,r1,r2,r3,r4}. 
Yet, to the best of our knowledge, it is the first time that this relation is studied in a reduced order model (ROM) setting for \ocp s. }In \cite{negri2013reduced}, a bound for the combined error of state, adjoint, and control variable is proposed. However, the strategy is based on an expensive computation of a lower bound to the Babu\v ska inf-sup constant of the optimality system. We remark that the computational costs to find the surrogate of the Babu\v ska inf-sup constant is
prohibitive even at the steady level. Our main effort is to avoid this issue, building a new global error estimator which can be efficiently evaluated during the basis construction for time dependent problems too. 
{The space-time discretization has practical limitations: increasing time and space resolution might not be feasible in terms of computational solution and storage costs. One can address this issue by increasing the computational resources available and using tailored preconditioning \cite{Benzi, schoberl2007symmetric}, multigrid strategies \cite{HinzeStokes,HinzeNS} or iterative methods applied to the space-time equations \cite{ManzoniFast}. These approaches can lighten
the computational costs of building the reduced model; however, their integration with the proposed framework goes beyond the scope of this contribution, i.e., providing a certified model of time-dependent problems based on a fast-to-compute error estimator.} 
To the best of our knowledge, the main contributions and findings can be summarized as follows: {
\begin{itemize}
\item[$\circ$] we build a new a posteriori lower bound for the Babu\v ska inf-sup constant based on quantities that can be inexpensive to compute. This allowed us to naturally extend the RB approach of \cite{negri2013reduced} also to time dependent problems, lightening the offline costs needed to build the reduced space.
\item[$\circ$] All the findings have been mirrored for the steady \ocp s, from the analytical to the numerical point of view, proposing our lower bound as an improvement to the approach followed in \cite{negri2013reduced}.
\end{itemize}}
The RB method has been tested on two Graetz flow models, with physical and geometrical parametrization {for distributed control}. Moreover, we compared the lower bound performances to the ones obtained with the employment of the Babu\v ska inf-sup constant, in terms of reliability and sharpness.
This work is outlined as follows. In Section \ref{gen_problem}, we describe the theoretical space-time formulation for \ocp s proposed in \cite{Strazzullo2}, adapting it to the structure of \cite{Langer2020}. Section \ref{FEM} introduces the space-time discretized system and the well-posedness of the finite-dimensional framework. Furthermore, we briefly analyze the algebraic system we deal with, following the all-at-once strategy of \cite{HinzeStokes,Stoll1,Stoll}. RB procedure is presented in Section \ref{sec_ROM}. First, we briefly introduce the Greedy approach following \cite{hesthaven2015certified}, then we present the bound already exploited in \cite{negri2013reduced}. Finally, we propose a new lower bound for the Babu\v ska inf-sup constant. Section \ref{results} shows the numerical results based on Graetz flow: a distributed OCP($\bmu$) with physical {and geometrical} parametrization. Conclusions follow in Section \ref{conclusions}.

%% file: ocp_formulation.tex
\section{Problem Formulation}
\label{gen_problem}
This Section aims at introducing linear quadratic \ocp s, governed by parabolic equations.  The main goal is to generalize and apply the space-time structure already presented for parabolic equations in \cite{urban2012new,yano2014space} and distributed optimal control problems in {\cite{HinzeStokes,HinzeNS,Langer2020}} to parametrized \ocp s governed by time dependent PDE($\bmu$)s with a general analysis of the well-posedness of the problem. First of all, we will focus our attention on time dependent problems, but then we will also specify the description to steady \ocp s, since our findings are still valid (with very few modifications) in the steady case.

\subsection{Time Dependent Problem}
\label{problem}
In this Section we introduce the continuous formulation of \ocp s governed by a parabolic state equation. We deal with a parametrized setting, where the parameter $\bmu \in \Cal P \subset \mathbb R^p$ could represent physical or geometrical features, with {$p \in \mathbb{N}$}.
Let {$\Omega \subset \mathbb R^n$, $n = 1, 2, 3$} be an open and bounded regular domain. {We denote with $\Gamma_D$ the portion of the {Lipschitz continuous} {\cite{quarteroni2008numerical} boundary $\partial \Omega$ where Dirichlet (homogeneous and/or non-homogeneous) boundary conditions are applied, while $\Gamma_N$ is the portion of the boundary $\partial \Omega$ that features Neumann boundary conditions with $\Gamma_D \cap \Gamma_N = \emptyset$ {and $\Gamma_D \cup \Gamma_N = \partial \Omega$} }.
{First all, without loss of generality, from now on we will consider only state problems featuring homogeneous Dirichlet boundary conditions. Indeed, the case of non-homogeneous boundary conditions can be recast in this framework after a lifting procedure \cite{quarteroni2008numerical}}.
The evolution of the system is studied in the time interval $(0,T)$ {for some $T > 0$}. 
{Let us define
$H^1_{\Gamma_D^0}(\Omega) = \{\, v \in H^1(\Omega) : v|_{\Gamma_D}=0\}$, where $v|_{\Gamma_D}=0$ is interpreted in the sense of traces.}
{We consider the separable Hilbert spaces $Y = H^1_{\Gamma_D^0}(\Omega)$ and $H = L^2(\Omega)$ which satisfy $Y \hookrightarrow H \hookrightarrow Y\dual$.}
Thus, we define {$\Cal {Y} = W(0,T)$}

 as the \emph{state space}.
{Furthermore, let us define the space $\Cal Q = L^2(0,T; Y)$.  
The aim of an \ocp $\,$is to {steer} a PDE($\bmu$) solution toward a desired observation $y_d(\bmu) \in L^2(0,T; Y_{obs})$, with $ Y_{obs} = L^2(\Omega_{obs})$, where $\Omega_{obs} \subseteq \Omega$ is the \emph{observation domain}.
The goal of reaching the desired state is achieved by a control variable $u \in \mathcal U = L^2(0,T; U)$ 
 where $U$ is a suitable Hilbert space depending on the action of the control. Indeed, we say that the control acts on the \emph{control domain}}. In the following, we only treat {the case of \emph{distributed} control}.
\\

\noindent {\textbf{Distributed Control}}. \\
When the control domain is $\Omega_u \subseteq \Omega$ and $U = L^2(\Omega_u)$, we are dealing with {distributed control}}.
{Since $\Omega_u$ is a subset of $\Omega$, it can intersect or even coincide with the observation domain.}\ The formulation of this OCP($\bmu$) reads as follows: for given {$x \in \Omega$, $t \in (0,T)$}, $\bmu \in \Cal P$ and forcing term
$f \in L^2(0,T; L^2({\Omega}))$, find the pair
{$(y,u) =(y(x,t,\bmu), u(x,t,\bmu)) \in \Cal Y \times \Cal U$} which solves
\begin{equation}
\label{eq_functional}
\min_{(y,u) \in {\Cal Y} \times \Cal U}
J((y,u); \bmu) = \frac{1}{2} \intTime{m(y - y_d(\bmu), y - y_d(\bmu); \bmu)}
+ \alf \intTime{n(u,u; \bmu)},
\end{equation}
governed by
{\begin{equation}
\label{eq_time_weak}
\begin{cases}
\displaystyle \intTime{s(y, q; \bmu)}
+ \intTime {a (y, q; \bmu)} =
\intTime{c(u ,q; \bmu)}
+\intTime{ \la G(\bmu), q \ra_{Y\dual, Y}} & \forall q \in \Cal Q, \\
y(0) = {y_0}  & \text{in } \Omega,
\end{cases}
\end{equation}
where $a \goesto{Y}{Y}{\mathbb R}$ is a bilinear form related to the state, 
$ G(\bmu) \in Y\dual $ is a continuous functional including forcing and boundary terms, and $y_0$ is the initial condition. Moreover:}
{
\begin{itemize}
\item[\small $\circ$]  $\displaystyle s(y, q; \bmu) = \left \la \dt{y}, q\right \ra_{{Y\dual , Y}}$,
\item[\small $\circ$] $c \goesto{U}{Y}{\mathbb R}$ is the bilinear form related to the distributed control case, i.e., $\displaystyle c(u,v; \bmu) = \int_{\Omega_u}uv \; dx$.
\item[\small $\circ$] $m \goesto {{Y_{obs}}}{{Y_{obs}}}{\mathbb R}$ and $n \goesto {U}{U}{\mathbb R}$ such that:
\begin{enumerate}
\item[(a)] \label{m_hyp} $m(\cdot, \cdot, \bmu)$ is a continuous with constant $c_m(\bmu)$, symmetric and positive semidefinite bilinear form, defined by the $L^2$ scalar product over the \emph{observation domain} $\Omega_{obs} \subseteq \Omega$,
\item[(b)] \label{n_hyp}$n(\cdot, \cdot, \bmu)$  is the scalar product of $U$ restricted to $\Omega_u$. Thus, when $\Omega_u=\Omega_{obs}$, i.e., $U=Y_{obs}$, the action of $n(\cdot, \cdot, \bmu)$ is equivalent to the one of $m(\cdot, \cdot, \bmu)$.
\end{enumerate}
\end{itemize}}

{Finally, $\alpha > 0$ is a fixed \emph{penalization parameter}}. We remark that the role of $\alpha$ influences the value of the control variable $u$: the larger is $\alpha$, the more the control will weight in the functional \eqref{eq_functional} and the less it will act on the system. {Moreover, we stress that $\alpha \leq 1$ is a common choice in many application fields (namely, the norm of the control weights, at most, as much as the state-related part of the functional) and it will always be the case in the numerical results we will deal with.}   
{A notable case, which will be covered in the numerical test cases in Section \ref{results}, is the one in which $\bmu$ contains geometrical parameters: without loss of generality, in our presentation we assume to have already traced back the problem to the \emph{reference domain} $\Omega$, and that the forms encode suitable pulled back operators. Namely, all the operators are traced back from the original geometric varying domain to the reference domain, still keeping their main features through an affine geometric transformation, see, e.g., \cite{RozzaHuynhPatera2008}}. 

We now make two other assumptions on the problem structure, i.e.,
\begin{enumerate}
\item[(c)] \label{coercivity_a} $a(\cdot, \cdot; \bmu)$ is continuous and coercive of constants $c_a(\bmu)$ and $\gamma_a(\bmu)$, respectively,
\item[(d)] \label{continuity_of_c} $c(\cdot, \cdot; \bmu)$ is continuous of constant $c_c(\bmu)$.
\end{enumerate}
We remark that hypotheses (c) and (d) ensure the existence of {a unique solution $y \in \Cal Y$ {to} \eqref{eq_time_weak}}, for a given $u \in \Cal U$ and $\bmu \in \Cal P$ {\cite{hinze2008optimization,negri2013reduced}}. {In a more compact notation: the weak {distributed} OCP($\bmu$) has the following form:  given $\bmu \in \Cal P$, find the pair $(y,u) \in \Cal Y \times \Cal U$ which satisfies
\begin{equation}
\label{general_problem}
\min_{(y, u) \in {\Cal Y} \times \Cal U} J((y, u); \bmu) \spazio \text{such that \eqref{eq_time_weak} holds}.
\end{equation}}
{Assumption (c) and (d) together with the function spaces definition and $\alpha > 0$ guarantees the existence of a unique solution $(y,u) \in \mathcal Y \times \mathcal U$ to problem \eqref{general_problem}: the interested reader may refer to \cite{hinze2008optimization}. Moreover, the assumptions we made, assure \cite[Corollary 1.3]{hinze2008optimization} to hold and, thus,
the proposed problem can be solved through a Lagrangian approach}. {To this end}, we define an \emph{adjoint variable}
{$p = p(\bmu) \in \Cal Y$}. 
{The assumption {$p \in \mathcal Y$} is not restrictive and it will be justified later in the manuscript in Remark \ref{rem_p}.}
For the sake of clarity, we underline that we work with the Hilbert space $Y$ also for the adjoint variable, rather then another Hilbert space, say $P$. The assumption $P \equiv Y$ is needed in order to guarantee the well-posedness of the problem at hand (it will be clearer to the reader in the space-time and RB settings).
In order to solve the minimization problem
\eqref{general_problem} we build the following Lagrangian functional:
\begin{align}
\label{functional}
 \Lg (y,u,p; \bmu) = J((y, u); \bmu)
+ \displaystyle \intTime{s(y, p; \bmu)} & + \intTime {a (y, p; \bmu)} \\ \nonumber
& \qquad  - \intTime{c(u,p; \bmu)}
- \intTime{ \la G(\bmu), p \ra_{Y\dual, Y}}.
\end{align}
To find the optimal pair $(y,u)$, we differentiate with respect to state, control and adjoint variables and, 
{{under ellipticity properties \cite{Hinze200545,r2}, exploiting the classical linear relation $\alpha u = p$ between control and adjoint variables in $\Omega_u$, we obtain the following optimality system: given $\bmu \in \Cal P$, find the pair $(y,p) \in {\Cal Y \times \Cal Y}$ such that
\begin{equation}
\label{global}
\Cal B((y, p), (z, q); \bmu) =\big \la \Cal F(\bmu), (z, q)\big \ra \quad \forall (z, q) \in \Cal Q \times \Cal Q,
\end{equation}
with
\begin{align}
\label{cal_B}
\Cal B((y, p), (z, q); \bmu) = & \intTime{ s(y, z; \bmu)}
+ \intTime{a(y, z; \bmu)} - \frac{1}{\alpha}\intTime{c(p, z; \bmu)}  \\ \nonumber
& \qquad + \intTime{m(y, q; \bmu)} - \intTime{s(p, q; \bmu)}+ \intTime{a(q, p; \bmu)},
\end{align}
and
\begin{equation}
\big \la \Cal F(\bmu), (z, q) \big \ra = \intTime{m(y_d, q; \bmu)} + \intTime{\big \la G(\bmu) , z \big \ra},
\end{equation}
{where, $\Cal Q = L^2(0,T; Y)$} {and $\langle \cdot,\cdot \rangle$ represents the duality pairing action}.}}\\

{
\begin{remark}
\label{rem_p}
We stress that the natural space for the adjoint variable $p$ is $\Cal Q$. However, considering it as a solution of the adjoint equation and using standard regularity arguments, the variable is actually in {$\mathcal Y$}, see e.g.\ \cite{hinze2008optimization}.
\end{remark}
}

{
\begin{remark}
    We here highlight the steady formulation of an \ocp, which be easily derived from the time dependent counterpart with the following modifications: 
    \begin{itemize}
        \item[\small $\circ$] the variables are $y \in Y, p \in Y$, and $u \in U$, while ${y_d} \in Y_{obs}$;
        \item[\small $\circ$] the steady problem reads:  given $\bmu \in \Cal P$, find the pair $(y,u) \in Y \times Y$ such that
\begin{equation}
\label{global_s}
\Cal B_s((y,p), (z,q); \bmu) = \left \la \Cal F_s (\bmu), (z,q) \right \ra  \quad \forall (z,q) \in Y \times Y.
\end{equation}
where $\Cal B_s(\cdot, \cdot; \bmu)$ and $\left \la \Cal F_s (\bmu), (z,q) \right \ra$ have an analogous definition with respect to \eqref{cal_B}, except for the time integration and the action of $s(\cdot, \cdot; \bmu)$. {The well-posedness of the problem is guaranteed by the same arguments of the time dependent case}.

    \end{itemize}
\end{remark}
}

Now, once provided with the continuous formulation of the problems, we can think about a discretization to simulate {the optimality system} for a given parameter $\bmu \in \Cal P$. In the next section, we will deal with the concept of \emph{high-fidelity approximation} of the system, adapting the space-time formulation proposed in \cite{Glas2017,HinzeStokes,hinze2008optimization,HinzeNS,urban2012new,yano2014space,yano2014} to the no-control saddle point structure of \cite{Langer2020}.

%% file: FE_OCP.tex
\section{Space-Time discretization: the High Fidelity problem}
\label{FEM}
This Section focuses on the discretization of the optimality system \eqref{global}. {The space-time approximation is a very intuitive discretization strategy already employed in several works for ROMs in parabolic PDEs \cite{Glas2017,urban2012new,yano2014space,yano2014}. In the context of PDE-constrained optimization, space--time formulations have been considered, for instance, in \cite{HinzeStokes,HinzeNS,Langer2020}. For error analysis for space--time discretizations of parabolic optimal control problems, the reader may refer to \cite{r1,r2,r3,r4}}.
{The main goal of this section is to provide the well-posedness of the discretized problem using Ne\v cas-Babu\v ska approach following the strategies in \cite{Langer2020}, {for the case $\Omega_u=\Omega_{\text{obs}}$}. \\ In literature, other approaches have been exploited to prove well-posedness results, in many settings: see \cite{r1,r2} for a priori error estimates for distributed control with and without control constraints, for example. Other existence and uniqueness result together with a priori error estimates for optimal control problems exploit the Petrov-Galerkin approximation, see e.g\ \cite{r3,r4,r5}. \\ In our setting, we employ the Finite Element (FE) approximation in space and {piecewise constant Discontinuous Galerkin (DG) approximation in time}. We propose an alternative approach to provide the well-posedness of the system. Indeed, the Theorem presented in this section will be exploited to derive an explicit error estimator to build the reduced space in Section \ref{sec_ROM}}.
Moreover, we will present the algebraic structure of the all-at-once space-time problem for time dependent and steady governing equations. The whole framework will be denoted as \emph{high-fidelity approximation} (HF), in contrast with the \emph{reduced basis approximation} (RB), which will be treated in Section \ref{sec_ROM}.

\subsection{{Space-time discretization} and well-posedness}
To solve the discretized version of \eqref{global}, first of all we define the FE function space $Y^{N_h} =
 Y \cap\Cal X_{\kappa}$, with
$$
 \mathcal X_{{\kappa}} = \{ s \in C^0(\overline \Omega) \; : \; s |_{K} \in { \mbb P^{\kappa}}, \; \; \forall K \in \Cal T  \},
$$
where $K$ is an element of a triangularization $\Cal T$ of the spatial domain $\Omega$ and ${\mathbb P^{\kappa}}$ is the space of polynomials of degree at most {$\kappa$}.  
 {We divide the time interval $I = (0,T]$ in $N_t$ subintervals $I_k = (t_{k-1}, t_k]$, with $k=1,\dots, N_t$ (more details will be provided in the next subsection), where $N_t$ represents the cardinality of the time discretization {and $\Delta t=T/N_t$}. We now define the following {DG(0)} space
$$
S_{N_t} = \{ s \in L^2(I) \; : \; s |_{I_k} \in \mbb P^0, \; \; 1 \leq k \leq N_t \}.
$$ 
{Specifically, we consider the same space-time space for state and adjoint variables, say
$
\mathcal Q\disc = S_{N_t} \otimes Y^{N_h},
$ of dimension is $\Cal N = N_h \cdot N_t$.} {This space has been widely used for the treatment of parabolic problems \cite{eriksson1985time2,eriksson1987error}, including for space-time reduced-order methods; see, e.g., \cite{yano2014,yano2014space}. Compared with fully coupled space-time finite element discretizations, it provides separate discretization for space and time, allowing for uncoupled refinement strategies and maintaining the affine structure for the problem.}
{To take into account the discontinuity of a DG function $w$, we define
$$
w\disc_{k_{+}} = \lim_{\epsilon \to 0^{+}} w\disc(t^{k} + \epsilon) \quad \text{and} \quad w\disc_{k_{-}} = \lim_{\epsilon \to 0^{+}} w\disc(t^{k} - \epsilon).
$$ From now on, we adopt the notation of \cite{eriksson1987error,eriksson1985time2} in the specific case of DG(0), where  $w\disc_{k_{-}}=w_{k}\disc$ and $w\disc_{k_{+}}=w\disc_{{k+1}}$ where $w_{k}\disc$ denotes the value of the discrete FE function in the time interval $I_k$. For more details, we refer the reader to \cite{eriksson1987error,eriksson1985time2}. 
The discretized problem is: given $\bmu \in \Cal P$ and observation $y_{\text{d}}\disc\in \mathcal Y_{\text{obs}}\disc = S_{N_t}  \otimes Y_{\text{obs}}$, find the pair $(y\disc, q\disc) \in \Cal Q \disc \times \Cal Q \disc$ such that

\begin{equation}
\label{global_FE}
\Cal B^{\mathcal N}((y\disc, p\disc), (z\disc, q\disc); \bmu) =\big \la \Cal F^{\mathcal N}(\bmu), (z\disc, q\disc)\big \ra \quad \forall (z\disc, q\disc) \in \Cal Q\disc \times \Cal Q\disc,
\end{equation}
with}

{
\begin{align}
\label{cal_B}
\Cal B\disc((y\disc, p\disc), (z\disc, q\disc); \bmu) = & { s\disc(y\disc, z\disc; \bmu)}
+ \Delta t\sum_{k=1}^{N_t}{a(y\disc_k, z\disc_k; \bmu)} - \frac{\Delta t}{\alpha} \sum_{k=1}^{N_t}{c(p\disc_k, z\disc_k; \bmu)}  \\ \nonumber
& \qquad +  \Delta t \sum_{k=1}^{N_t}{m(y\disc_k, q\disc_k; \bmu)} +{ {s\disc}^{\ast}(p\disc, q\disc; \bmu)}+ \Delta t \sum_{k=1}^{N_t}{a(q\disc_k, p\disc_k; \bmu)},
\end{align}
and
\begin{equation}
\big \la \Cal F\disc(\bmu), (z\disc, q\disc) \big \ra = \Delta t \sum_{k=1}^{N_t}{m({y{\disc_d}}_k, q\disc_k; \bmu)} + \Delta t \sum_{k=1}^{N_t}{\big \la G(\bmu) , z\disc_k \big \ra} + (y\disc_0, z_1\disc)_H,
\end{equation}
where we take into consideration forcing terms, boundary conditions and discretized FE initial condition $y\disc_0$. The initial condition for the state and the final condition for the adjoint have been imposed weakly as in \cite{yano2014}. {The} left-hand side of the weak imposition of the initial and final conditions is considered in the definition of $s\disc(\cdot, \cdot; \bmu)$ and ${s\disc}^{\ast}(\cdot, \cdot; \bmu)$, given by {
\begin{align}
\label{eq:DGevo}
s\disc(y\disc, z\disc; \bmu) 
& = \sum_{k=1}^{N_t} (y\disc_k, z\disc_k)_H - \sum_{k=2}^{N_t}( y\disc_{k-1}, z\disc_k)_H ,
\end{align}
and
\label{eq:DGevo}
\begin{align}
{ {s\disc}^{\ast}(p\disc, q\disc; \bmu)} 
& = \sum_{k=2}^{N_t + 1} (p\disc_{k-1}, q\disc_{k-1})_H - \sum_{k=2}^{N_t} (p\disc_{k}, q\disc_{k-1})_H. 
\end{align}}
For the sake of the presentation, we do not report the adjoint computations, for details see, e.g., \cite{adj,eriksson1987error,eriksson1985time2}. We recall that by construction the final condition of the adjoint is always zero. For this reason, no term appears in the definition of $\mathcal F\disc(\bmu)$.\\
From the formulation of \eqref{global_FE}, focusing on one element of the summation of $I_k$ and following directly  \cite{eriksson1987error,eriksson1985time2}, we note that the DG(0) approximation is equivalent to a {backward Euler scheme for the state and a forward Euler scheme for the adjoint
variables}. More details on the algebraic formulation are postponed in Section \ref{sec:alg}.} 
{Moreover, 
at the discrete level, we define the norm $${\norm{{y\disc}}_{\mathcal Q \disc}^2 =  \Delta t \sum_{k=1}^{N_t} \norm{y\disc_k}_{Y}^2.}$$}
We now prove the well-posedness of the space-time problem \eqref{global_FE} {based on an application of the} Ne\v{c}as-Babu\v{s}ka theorem \cite{necas} {by assuring both a discrete} inf-sup stability condition in the discretized space $\Cal Q \disc$ and a continuity property of the discrete system. The latter is trivial.

Now, it still remains to prove the inf-sup stability condition with respect to the discretized spaces. 

\begin{lemma}[{Discrete} Surjectivity of ${\Cal B^{\mathcal N}}^*$] Assuming $\Omega_u = \Omega_{\text{obs}}$, {the bilinear form $\mathcal B^{\mathcal N}$ satisfies} the following inf-sup stability condition: {denoting with $\mathcal Q_0\disc$ the space $(\Cal Q \disc\times \Cal Q\disc)\setminus \{(0,0)\}$, there exists $\beta\disc(\bmu) > 0$ such that}
\label{lemma_surj_FE}
\begin{equation}
\label{surj_FE}
\beta_{\Cal B}\disc(\bmu) \eqdot \inf_{(y\disc,p\disc) \in \Cal Q\disc_0}\sup_{(z\disc,q\disc) \in  \Cal Q_0\disc} \frac{\Cal B\disc((y\disc,p\disc), (z\disc, q\disc); \bmu)}
{ \sqrt{\norm{y\disc}_{\Cal Q\disc}^2 + \norm{p\disc}_{\Cal Q\disc}^2}\sqrt{\norm{z\disc}_{\Cal Q\disc}^2 + \norm{q\disc}_{\Cal Q\disc}^2}} \geq \beta\disc(\bmu) > 0.
\end{equation}
\end{lemma}

\begin{proof}
We recall the equality $c(p\disc, y\disc; \bmu) = m(y\disc, p\disc; \bmu)$ (deriving from the assumption $\Omega_{u} = \Omega_{\text{obs}}$ and the definitions of Section \ref{gen_problem}) and the following relation for any $w\disc \in \Cal Q \disc$
\begin{align}
\label{eq:s_T}
{s^{\mathcal N}(w\disc, w\disc; \bmu)} & \geq \frac{1}{2}
\left ( \norm{w\disc_1}^2_H + \norm{w\disc_{N_t}}^2_H \right ) \geq 0.
\end{align}
The relation \eqref{eq:s_T} is easily provable with simple manipulation over summation and applying Young's inequality.
Analogously, it holds 
\begin{align} { 
\label{eq:s_0}
{s^{\mathcal N}}^{\ast}(w\disc, w\disc; \bmu)} 
&\geq \frac{1}{2}
\left ( \norm{w\disc_1}^2_H + \norm{w\disc_{N_t}}^2_H \right ) \geq 0.
\end{align} 
In this way, specifying
$(z\disc,q\disc) \equiv (\alpha y\disc,p\disc)$, since $\Omega_u = \Omega_{\text{obs}}$ we have 
$$
\Cal B\disc((y\disc,p\disc), (\alpha y\disc,p\disc); \bmu) = \alpha {s\disc}(y\disc,y\disc) + \alpha \Delta t\sum_{k=1}^{N_t}{a(y\disc_k, y\disc_k; \bmu)} +{ {s\disc}^{\ast}(p\disc, p\disc; \bmu)}+ \Delta t \sum_{k=1}^{N_t} a(p\disc_k, p\disc_k; \bmu).
$$
Thus, {exploiting also the coercivity of $a(\cdot, \cdot; \bmu)$in the second term and in the last term of the system, specifically, $a(y_k^{\mathcal N},y_k^{\mathcal N}; \bmu) \geq \gamma_a(\bmu)\norm{y_k^{\mathcal N}}_{Y}^2$ and $a(p_k^{\mathcal N},p_k^{\mathcal N}; \bmu) \geq \gamma_a(\bmu)\norm{p_k^{\mathcal N}}_{Y}^2$}, we obtain
\begin{align*}
\label{u_equal_obs}
\nonumber
\sup_{(z\disc,q\disc) \in \Cal Q_0\disc}
\frac{\Cal B\disc((y\disc,p\disc), (z\disc,q\disc); \bmu)}{ \sqrt{\norm{y\disc}_{\Cal Q\disc}^2 + \norm{p\disc}_{\Cal Q\disc}^2}\sqrt{\norm{z\disc}_{\Cal Q\disc}^2 + \norm{q\disc}_{\Cal Q\disc}^2}}
& \geq \frac{\Cal B\disc((y\disc,p\disc), (\alpha y\disc,p\disc); \bmu)}{ \sqrt{\norm{y\disc}_{\Cal Q\disc}^2 + \norm{p\disc}_{\Cal Q\disc}^2}\sqrt{\norm{\alpha y\disc}_{\Cal Q\disc}^2 + \norm{p\disc}_{\Cal Q\disc}^2}}\\
& \geq \frac{ \alpha \gamma_a(\bmu)\norm{y\disc}_{\Cal Q\disc}^2 + {\gamma_a}(\bmu) \norm{p\disc}_{\Cal Q\disc}^2}{\sqrt{(\max\{\alpha^2,1\})}\left (\norm{y\disc}_{\Cal Q\disc}^2 + \norm{p\disc}_{\Cal Q\disc}^2\right)} \\
& \geq \frac{\min \{\alpha \gamma_a(\bmu), \gamma_a(\bmu) \} (\norm{y\disc}_{\Cal Q\disc}^2 + \norm{p\disc}_{\Cal Q\disc}^2)}{\sqrt{\max\{\alpha^2,1\}}\left (\norm{y\disc}_{\Cal Q\disc}^2 + \norm{p\disc}_{\Cal Q\disc}^2\right)} \\ \nonumber
& \geq  \frac{\min \{\alpha \gamma_a(\bmu), \gamma_a(\bmu) \}}{\sqrt{\max \{\alpha^2,1\}}}.
\end{align*}

Since the proved relation does not depend on the choice of $(y\disc,p\disc) \in \Cal Q\disc \times \Cal Q\disc$, then the inf-sup condition \eqref{surj_FE} {holds}. {We recall that at the discrete level we are not considering the full norm of $\mathcal Y$.}
\end{proof}

In the case where $\Omega_u = \Omega_{\text{obs}}$, we have all the ingredients to prove the well-posedness of the space-time discretized problem \eqref{global_FE}, indeed, the finite dimensional Ne\v cas-Babu\v ska holds due to continuity of \eqref{global_FE} and Lemma \ref{lemma_surj_FE} since, at the discrete level, the inf-sup stability of \eqref{surj_FE} is equivalent to the following inf-sup {stability}  (the interested reader may refer to \cite{Babuska1971,xu2003some})
\begin{equation*}
 {\beta_{\Cal B}\disc(\bmu) \eqdot  \inf_{(z\disc,q\disc) \in Q_0\disc}\sup_{(y\disc,p\disc) \in  Q_0\disc} \frac{\Cal B\disc((y\disc,p\disc), (z\disc, q\disc); \bmu)}
{ \sqrt{\norm{y\disc}_{\Cal Q\disc}^2 + \norm{p\disc}_{\Cal Q\disc}^2}\sqrt{\norm{z\disc}_{\Cal Q\disc}^2 + \norm{q\disc}_{\Cal Q\disc}^2}} \geq \beta(\bmu) > 0}.
\end{equation*}
As a consequence, we can state the following theorem:
\begin{theorem}
\label{theorem_right}
Assuming $\Omega_u = \Omega_{obs}$, for a given $\bmu \in \Cal P$ and observation ${y_{\text{d}}^{\mathcal N}} \in \mathcal Y_{\text{obs}}\disc$, the problem \eqref{global_FE} has a unique solution in ${(y\disc, p\disc)} \in \Cal Q\disc \times \Cal Q \disc$.
\end{theorem}

\color{black}
\begin{remark}[Steady Case]
In the steady cases, we deal with a standard FE discretization, i.e. $\Cal N = N_h$. In this case, the optimization problem reads: given $\bmu \in \Cal P$, and an observation in
$Y_{\text{obs}}^{N_h}$, find the pair $(y\disc,p\disc) \in Y^{N_h} \times Y^{N_h}$ such that
\begin{equation}
\label{global_s_FE}
\Cal B_s((y\disc,p\disc), (z\disc,q\disc)) = \left \la \Cal F_s (\bmu), (z\disc,q\disc) \right \ra  \quad \forall (z\disc,q\disc) \in  Y^{N_h} \times Y^{N_h}.
\end{equation}
Also in this case we can state its well-posedness. To do so, we need the following Lemma, after denoting $Y_0^{N_h}$ the space $(Y^{N_h} \times Y^{N_h})\setminus \{(0,0)\}$.
{
\begin{lemma}[Surjectivity of $\Cal B_s\dual$]
\label{lemma_surj_s}
{Assuming $\Omega_u=\Omega_{\text{obs}}$}, the bilinear form \eqref{global_s} {satisfies} the following inf-sup stability condition: {there exists $\beta_s(\bmu) > 0$ such that}
\begin{equation}
\label{surj_s}
\beta_{{\Cal B}_s}^{N_h}(\bmu) =
\inf_{(y\disc, p\disc) \in Y_0^{N_h}}\sup_{(z\disc, q\disc) \in Y_0^{N_h}} \frac{\Cal B_s((y\disc,p\disc), (z\disc, q\disc); \bmu)}
{ \sqrt{\norm{y\disc}_{Y}^2 + \norm{p\disc}_{Y}^2}\sqrt{\norm{z\disc}_{Y}^2 + \norm{q\disc}_{Y}^2}} \geq \beta_s(\bmu).
\end{equation}
\end{lemma}
\begin{proof}
Choosing $z\disc = \alpha y\disc$ and $q\disc = p\disc$ leads to
\begin{equation}
\label{estimate_1}
\sup_{(z\disc, q\disc) \in Y_0^{N_h}} \frac{\Cal B_s((y\disc,p\disc), (z\disc, q\disc); \bmu)}
{ \sqrt{\norm{y\disc}_{Y}^2 + \norm{p\disc}_{Y}^2}\sqrt{\norm{z\disc}_{Y}^2 + \norm{q\disc}_{Y}^2}} \geq \frac{\min \{\alpha \gamma_a(\bmu), \gamma_a(\bmu) \}}{\sqrt{(\max \{\alpha^2,1\})}},
\end{equation}
since under the assumption of coincidence of control and observation domain $c(p\disc, y\disc;\bmu)$ is the same as $m(y\disc,p\disc;\bmu)$. For the arbitrariness of $y\disc$ and $p\disc$, the inf-sup condition \eqref{surj_s} holds. 
\end{proof}
}
We can now prove the following theorem:
\begin{theorem}
\label{s_FE}
Assumimg $\Omega_u = \Omega_{\text{obs}}$, for a given $\bmu \in \Cal P$ and for a given observation $y_{\text{d}} \in Y_{\text{obs}}^{N_h}$, problem \eqref{global_s_FE} has a unique solution $(y\disc,p\disc) \in Y^{N_h} \times Y^{N_h}$.
\end{theorem}
\begin{proof}
Once again, the statement is a simple consequence of the Ne\v cas-Babu\v ska theorem: indeed, the continuity is directly inherited from the continuous formulation of the bilinear forms at hand and we can use the same arguments of Lemma \ref{lemma_surj_s} to prove
$\beta_{{\Cal B}_s}^{N_h}(\boldsymbol \mu) > 0.$
\end{proof}
\end{remark}
{
\begin{remark}[Extensions]
Theorems  \ref{theorem_right} and \ref{s_FE} can be extended to more generic scenarios as $\Omega_u \neq \Omega_{\text{obs}}$ and boundary control. However, it would require some restrictions on the parametric space $\mathcal P$ {to guarantee coercivity, leading to additional analytical and numerical challenges}. Since we want to deal with a generic parametric setting in Section \ref{sec_ROM}, we decided to focus on the specific case of the control and observation domains coinciding. 
\end{remark}
\begin{remark}[DG(0) and Euler time stepping]
As specified in \cite{eriksson1987error,yano2014space} the proposed space-time formulation based DG(0) time discretization is equivalent to an Euler discretization for the time evolution.
 For clarity, in Section \ref{sec:alg}, we refer to the system as Euler-based discretized in time. This approach was already exploited, for example, in \cite{Langer2020,Strazzullo2,Strazzullo3} and \cite{HinzeNS,HinzeStokes} for space-time multigrid methods for Stokes and Navier-Stokes equations, where the Galerkin discretization {for both state and adjoint variables is applied} in space and time. \\
Other choices can be made, based on different discretizations in time for state and adjoint variables, see for example the Petrov-Galerkin analysis of \cite{Glas2017,urban2012new,yano2014space,yano2014} for parabolic problems in forward frameworks and of \cite{r3,r4,r5} in the optimal control setting.
\end{remark}
}
\subsection{Space-Time All-At-Once Algebraic System}
\label{sec:alg}
We now specify the algebraic {formulation of the space-time discretization}. We used a space FE approximation using ${ \mathbb P^{1}}-{ \mathbb P^{1}}$ pair for state and adjoint variables. {For what concerns} time discretization,  we divide the time interval $[0,T]$ in $N_t$ equispaced subintervals of length $\Delta t$. As already specified, at each {time-sub-interval $I_k$}, we can consider the variables $y\disc_k$ and $p\disc_k$ in $Y^{N_h}$ {to be} represented through FE basis functions $\{\phi^i\}_{i=1}^{N_h}$ as follows
$$
y\disc_k = {\sum_{1=1}^{N_h}}\bar y^{i}_k \phi_{i}, \hspace{1cm}
p\disc_k = {\sum_{1=1}^{N_h}}{\bar p^{i}_k} \phi_{i}.
$$
{We now define the space-time state and adjoint vectors  $\mathsf y = [\bar y_1, \dots, \bar y_{N_t}]^T$ and $\mathsf p = [\bar p_1, \dots, \bar p_{N_t}]^T$. In the same fashion, let $\mathsf y_0 = [\bar y_0, 0, \dots, 0]^T$, $\mathsf y_d = [\bar {y_d}_1, \bar {y_d}_2, \dots, \bar {y_d}_{N_t}]^T$ and $\mathsf f = [\bar f_1, \dots, \bar f_{N_t}]^T$ be the space-time vectors describing the weakly imposed state initial condition, the desired solution profile and the forcing term, respectively. {All the vectors are considered in $\mathbb R^{N_h \cdot N_t}$.} In other words, $\bar y_k$ and $\bar p_k$ are the column vectors of the FE element coefficients of the state and adjoint variables in the sub-interval $I_k$, for $k=1, \dots, N_{t}$.} 
We aim to recast the optimality system \eqref{global_FE} in the classical saddle point structure already presented in \cite{HinzeStokes,Stoll1,Stoll,Strazzullo2}, with the necessary modifications due to the elimination of the direct computation of the control variable. \\
For the sake of clarity, we focus our attention on the state equation. {First}, we introduce $\mathsf S(\bmu)$ and $\mathsf D_a(\bmu)$ {in $\mathbb R^{N_h \times N_h}$}, which are defined as $\mathsf S(\bmu)_{ij} = s(\phi_j, \phi_i; \bmu)$ and $\mathsf D_a(\bmu)_{ij} = a(\phi_j, \phi_i; \bmu)$, for $i,j = 1, \dots, N_h$, respectively. For the sake of notation, a {subscript} will indicate {the non-zero entries of a matrix corresponding to an action on a prescribed domain}, {the notation $\mathsf M_{\Omega_{\text{obs}}(\bmu)}$ shall denote the matrix in $\mathbb R^{N_h \times N_h}$ representing the mass matrix in the degrees of freedom of the observation domain, with zero entries elsewhere}, for example. Furthermore with $ {\mathsf {C}_{\Omega_u}(\bmu) \in \mathbb R^{N_h \times N_h}}$ we indicate the matrix which satisfies ${\mathsf {C}_{\Omega_u}}(\bmu)_{ij} = c(\phi_j, \phi_i; \bmu)$, for {the basis functions having support in $\Omega_u$, while the entries are zero elsewhere}. Concerning time, as we said, we follow the DG(0) discretization.

{
Exploiting the classical space-time approximation strategies, 
{see, e.g., \cite{HinzeStokes,hinze2008optimization,Stoll,Stoll1} for an overview on the following algebraic formulation},
the space-time state equation can be written in compact form as
\begin{equation}
\mathsf K(\bmu) \mathsf  y - \frac{\Delta t}{\alpha} \mathsf C(\bmu) \mathsf p = \mathsf S(\bmu) \mathsf y_0 + \Delta t  \mathsf f,
\end{equation}
where $\mathsf C(\bmu)$ and $\mathsf K(\bmu)$ are the block-diagonal matrices of dimension  $ \mathbb R^{N_h  \cdot N_t} \times \mathbb R^{N_h  \cdot  N_t}$ which entries are given by ${\mathsf {C}_{\Omega_u}}(\bmu)$ and $\mathsf S(\bmu) + \Delta t \mathsf D_a(\bmu)$, respectively. A similar argument can be applied to discretize the {space-time} adjoint equation: 
\begin{equation}
\mathsf K(\bmu)^T \mathsf p + \Delta t \mathsf M(\bmu) \mathsf y = \Delta t \mathsf M(\bmu) \mathsf y_d,
\end{equation}
where $\mathsf M(\bmu) \in \mathbb R^{N_h  \cdot N_t} \times \mathbb R^{N_h  \cdot  N_t}$ is the block-diagonal matrix which entries are given by ${\mathsf M_{\Omega_{obs}}}(\bmu)$.
}}

If we combine the information given by the two equations, we  end up with the following all-at-once system:
\begin{equation}
\label{one_shot}
\begin{bmatrix}
\Delta t  \mathsf M(\bmu) &  \mathsf {K}(\bmu)^T  \\
\mathsf K(\bmu) & -\frac{\Delta t}{\alpha} \mathsf C(\bmu) \\
\end{bmatrix}
\begin{bmatrix}
\mathsf y \\
\mathsf p \\
\end{bmatrix}
=
\begin{bmatrix}
\Delta t \mathsf {M(\bmu)} \mathsf y_d \\
\mathsf {S(\bmu)} \mathsf y_0 + \Delta t \mathsf f \\
\end{bmatrix}.
\end{equation}

{It is well-known in literature that optimal control problems solved through Lagrangian formulation result in saddle point structures. This structure has been highlighted extensively in ROMs for stationary optimal control problems, see e.g., \cite{tesi} for an overview. Saddle point time dependent formulations have been investigated in, e.g., in \cite{Stoll1,Stoll,HinzeStokes,HinzeNS} and extended to parametric setting and ROMs in \cite{Strazzullo2,Strazzullo3}}.
Also in the no-control framework, we see that the structure is preserved. \\
In this Section, we have shown that in order to solve the all-at-once space-time optimization in a parametrized setting, for a given $\boldsymbol \bmu \in \Cal P$, we have to deal with a system whose global dimension is $2\Cal N$. If we consider a context where the optimal solution is studied for several parameters, the {required} computational resources grow and {solving} the problem can {result in an unbearable amount of simulation time}, due to the high dimensionality of discretized space-time structure presented, which drastically increases if mesh refinement has to be performed in space and/or in time. 
In the next Section, in order to manage the issue of the high computational costs, we propose {to use the certified} reduced basis (RB) method as a strategy to reliably solve \ocp s in a smaller amount of time. After a general introduction of the basic ideas behind RB methods, we propose a certified error estimator specifically built for such parabolic optimal control problems in order to apply a greedy algorithm in the reduction  procedure.

%% file: RB_OCP.tex
\section{RB for parabolic time dependent parametric optimal control problems}
\label{sec_ROM}
This Section describes reduced approximation for time dependent \ocp s exploiting a greedy algorithm as presented in
\cite{hesthaven2015certified,RozzaHuynhPatera2008}. First of all, we will introduce the main framework of RB and its applicability thanks to the affine assumption. Then we will provide a new error estimator suited for both time dependent and steady problems. Finally, we will briefly present the reduced saddle point optimization problem and the \emph{aggregated space strategy}  following the approach {already} exploited in {previous} literature, see e.g. { \cite{bader2015certified,gerner2012certified,karcher2018certified,negri2015reduced,negri2013reduced} for steady optimal control problems in a greedy fashion. {Other applications of greedy strategies to control systems can be found in \cite{HinzeNew,Hendrik}, while we refer the reader to \cite{kunisch2008proper} for the use of Proper Orthogonal Decomposition (POD) in a non-parametric control setting and to \cite[Section 2.6]{Hinzelibro} for a survey on the application of POD to parametric optimal control problems.}  \\ {All of the concepts} will be presented for the time dependent scenario: indeed, the steady case complies with the space-time more general formulation. 
We remark that our goal was to extend the results for steady problems of \cite{negri2013reduced} to the space-time setting, following the strategies introduced for forward parabolic problems in \cite{urban2012new,yano2014space}. However, other techniques can be employed, not necessarily relying on the entire space-time coupled system in a Ne{\v c}as-Babu{\v s}ka framework, see e.g. \cite{dede2010reduced,Karcherrev}, where the space is built in an adaptive way for each time-step. Our approach, in contrast, naturally fits steady scenarios, even improving the results of \cite{negri2013reduced}.

}

\subsection{Reduced Problem Formulation}
\label{off_on}
{In Section \ref{problem}, we presented linear quadratic time dependent \ocp s in no-control framework \eqref{global}. For the sake of clarity, in this section, we will always make the parameter dependence explicit: this will help us in being clearer in describing the RB fundamentals, following the notions in \cite{hesthaven2015certified}. \\
The goal of reduced strategies is to build a low-dimensional surrogate space $\Cal Q_N \times \Cal Q_N \subset \Cal Q \disc \times \Cal Q \disc \subset {\Cal Y \times \Cal Y}$, assuming that the solution changes smoothly with respect to the parameter $\boldsymbol \mu \in \mathcal P$}. The main feature of  the reduced space is that it is spanned by some properly chosen \emph{snapshots}, i.e., HF solutions computed for some values of the parameter $\bmu \in \Cal P$. {Once the reduced spaces are built}, a Galerkin projection is performed in order to find the optimal solution for a given parameter realization.
In other words, {after} a possibly costly reduced space construction phase, the RB strategy allows to solve a low dimensional system at each new parametric instance, in a space of dimension $N \ll \Cal N$.
In our case, the reduced optimality system for a given observation reads: given $\bmu \in \Cal P$ find the optimal pair $(y_N(\bmu), p_N(\bmu)) \in \Cal Q_N \times \Cal Q_N$ such that
\begin{equation}
\label{RB_OCP}
\Cal B \disc((y_N(\boldsymbol{\mu}), p_N(\boldsymbol{\mu})),(z_N, p_N); \bmu) =
\la \Cal F \disc(\bmu), (z_N, q_N) \ra \qquad \forall (z_N, q_N) \in \Cal Q_N \times \Cal Q_N.
\end{equation}
It is clear that the presented approach is convenient only if you can solve \eqref{RB_OCP} rapidly and independently from the value $\Cal N$. {This is guaranteed by the affinity of the problem, which allows for the division the {reduced} space construction and the solution process in two separate steps:
\begin{itemize}
\item[$\small{\circ}$]an \emph{offline phase} where the reduced spaces are built and all the $\bmu-$independent quantities are stored. This part of the process can be expensive, but it is performed only once.
\item[$\small{\circ}$] An \emph{online phase} which solves the reduced system at each new parameter evaluation. In other words, the Galerkin projection in \eqref{RB_OCP} can be performed for several values of $\bmu \in \Cal P$  in a small amount of time.
\end{itemize}}

We still need to describe explicitly how to construct the reduced function spaces: it will be the content of the next section.

\subsection{Greedy Algorithm and Space Construction}

\label{Greedysec}
In this section, we present the algorithm we used in order to build the reduced spaces. We rely on \emph{greedy algorithm}, see \cite{buffa2012priori,hesthaven2015certified} for a general introduction. {For time dependent problems, the vast majority of the reduced basis literature relies on the POD-Greedy algorithm \cite{haasdonk2008reduced,hesthaven2015certified}, that is usually called upon to deal with the compression of solution instances coming from a time stepping scheme. Instead, in the present space-time formulation, we are going to extend what has been already done for the space-time formulation of} parabolic problems in \cite{urban2012new,yano2014space} to time dependent linear parabolic \ocp s. \\
The greedy approach is an iterative technique based on the idea to add new information to the reduced spaces at each {enrichment step by adding a suitably} chosen snapshot. At each iteration a HF solution of the optimality system \eqref{global_FE} (or \eqref{global_s_FE} for the steady case) is needed, i.e. in order to build a $N-$dimensional reduced space, $N$ space-time (or steady) solutions must be evaluated. Before describing the algorithm, let us define the global error $e$ between the discrete optimality system and the reduced one, i.e.,
$
e := (y\disc - y_N, p\disc - p_N).
$
{To efficiently apply the greedy algorithm, we have to \emph{estimate} the norm of the global error independently from the HF dimension

through an estimator  $\Delta_N(\bmu)$. }

In the next Section we will provide an explicit expression for $\Delta_N(\bmu)$. For now, {we assume we already have such an estimator}.
{The first step is to choose a discrete subset of parameters $\mathcal P_h \subset \mathcal P$ of cardinality $N_{\text{max}} = |\mathcal P_h|$}. 
We fix a tolerance $\tau$ and initialize the reduced spaces $\Cal Q_N^y = \text{span}\{y\disc(\bmu_0)\}$ and $\Cal Q_N^p = \text{span}\{p\disc(\bmu_0)\}$ for {an} initial $\bmu_0$.
The $n-$th step of the process we choose the parameter that maximizes the estimator 

and, then, we enrich the spaces with the snapshots related to the new {selected} parameter. 
{We proceed with the iterations until the estimator verifies the following stopping criterion for a picked parameter $\bmu_n$: 
$\Delta_N(\bmu_n)^0\leq \tau$, where $\Delta_N(\bmu_n)^0$ is the error estimator normalized with respect to the error estimator value at the first iteration.}
At the end of the algorithm, we will be provided by two spaces which {separately} describe state and adjoint variables. We remark that in order to have the space-time formulation well-posed, the state and adjoint space must be discretized through the same technique. It is intuitive to assert that the greedy procedure does not lead to the same space for state and adjoint, since the snapshots for those variables can span different function spaces. {To solve this issue, at each iteration of the greedy algorithm, we employ the aggregated space technique, a classical approach for saddle point problem arising in PDE($\bmu$)-constrained optimization \cite{bader2015certified,dede2010reduced,gerner2012certified,karcher2014certified,karcher2018certified,kunisch2008proper,negri2015reduced,negri2013reduced}. It consists in the definition of the reduced space
$\Cal Q_N = \Cal Q_N^y \cup \Cal Q_N^p$,
that will be {employed} in the representation of both the state and the adjoint variables, i.e.,
given $\bmu \in \Cal P$, $y_N(\bmu) \in \Cal Q_N$ and $p_N(\bmu) \in \Cal Q_N$.}
Even if this strategy increases the dimension of the reduced problem from $2N$ to $ 4N$, for time consuming problems such as parametrized time-dependent \ocp s, it is still convenient {to solve} several reduced problems {as opposed} to the space-time discretization counterpart.
The next section presents a specific estimator  $\Delta_N(\bmu)$, which can be used both for time dependent and steady \ocp s.
\subsection{Rigorous a posteriori error estimate}
\label{error_estimate}
{In the context of RB method, an a posteriori error estimate is needed for a {reliable} reduction algorithm. 
Our analysis {is based on} an a posteriori error estimate which complies with the classical
Brezzi's and Ne{\v c}as-Babu{\v s}ka stability theories \cite{Babuska1971,boffi2013mixed,necas}, well known in the context of steady \ocp s, e.g. see \cite{negri2015reduced,negri2013reduced}. \\
As already specified, our aim is {to build} a posteriori error bound $\Delta_N(\bmu)$ which is $\Cal N-$independent and verifies:
\begin{equation}
\norm{e}_{\Cal Q\disc \times \Cal Q\disc} = \sqrt{\norm{ y\disc  - y_N}_{\Cal Q\disc}^2 + \norm{ p\disc - p_N}_{\Cal Q\disc}^2}
\leq \Delta_N(\bmu),
\end{equation}
so that the error bound is fast to compute}. \\
At the space-time {discrete} level, we have {that the inf-sup condition \eqref{surj_FE} is} verified and, {after application of} the Ne\v cas-Babu\v ska theorem, {following standard techniques exploited in \cite{negri2013reduced} we have} 
$$
\norm{e}_{\Cal Q\disc \times \Cal Q\disc} \leq \frac{\norm{\Cal R}_{(\Cal Q \times \Cal Q)\dual}}{\beta_{\Cal B} \disc(\bmu)} \spazio \forall \bmu \in \Cal P,
$$ where
\begin{equation}
\Cal R((z, q); \bmu) = \Cal B\disc((y_N, q_N), (z\disc, q\disc); \bmu) - \la \Cal F\disc(\bmu), (z\disc, q\disc) \ra
\spazio \forall (z\disc, q\disc) \in {\Cal Q \disc \times \Cal Q \disc}.
\end{equation}
The main {issue in the practical application of the} estimate is to find an effective way to compute the inf-sup constant $\beta_{\Cal B}\disc(\bmu)$. Let us suppose to have at one's disposal a {lower} bound $\beta^{LB}(\bmu)>0$ such that
$\beta_{\Cal B}\disc(\bmu) \geq \beta^{LB}(\bmu)$. Then the error estimate becomes, for any $N$:
\begin{equation}
\label{error}
\norm{e}_{\Cal Q \disc \times \Cal Q\disc} \leq \frac{\norm{\Cal R}_{(\Cal Q \times \Cal Q)\dual}} {\beta^{LB}(\bmu)}
\eqdot \Delta_N(\bmu)  \spazio \forall \bmu \in \Cal P.
\end{equation}
We now will give an analytical formulation to the lower bound $\beta^{LB}(\bmu)$.
\begin{theorem}
\label{th_lb}
Let us suppose that a space-time OCP($\bmu$) {with coinciding control and observation domains} is well-posed. Then we obtain a lower bound for $\beta^{\Cal N}_{\Cal B} (\bmu)$ of the form:
\begin{equation}
\label{bounds}
    \beta^{LB}(\bmu) = \frac{\min \{\alpha \gamma_a(\bmu), \gamma_a(\bmu)\}}{\sqrt{\max\{\alpha^2,1\}}}.
\end{equation}
\end{theorem}
\begin{proof}
{
The statement is a consequence of Theorem \ref{theorem_right} and classical arguments for error estimation (see e.g. \cite{hesthaven2015certified,negri2013reduced}) applying the Ne\v cas-Babu\v ska stability estimate to the equation
$
\Cal B^{\mathcal N}(e, (z\disc, q\disc); \bmu) = \Cal R((z\disc, q\disc); \bmu).
$}
\end{proof}
It is clear that $\beta^{LB}(\bmu)$ is very fast to be computed for a given $\bmu \in \Cal P$ due to the nature of all the constants involved, which are independent of the space-time formulation. Furthermore, it is well known that the computation of the dual norm of the residual can be efficiently evaluated thanks to the affine assumption {through suitable Riesz representers} \cite{negri2015reduced,negri2013reduced,RozzaHuynhPatera2008}. 
\begin{remark}[The steady case]
All the arguments on this Section can be applied to steady linear \ocp s. In this case, the reduced problem reads: given $\bmu \in \Cal P$ find the optimal pair $(y_N(\bmu), p_N(\bmu)) \in Y_N \times Y_N$ such that
\begin{equation}
\label{RB_OCP_s}
\Cal B_s((y_N(\boldsymbol{\mu}), p_N(\boldsymbol{\mu})),(z, p), \bmu) =
\la \Cal F_s(\bmu), (z_N, q_N) \ra \spazio  \forall (z, q) \in Y_N \times  Y_N,
\end{equation}
with $Y_N \subset Y^{N_h}$. It is well known, that there exists a positive inf-sup steady stability constant $\beta_{\Cal B_s}\disc (\bmu)$ and  under the affine decomposition assumption, one can efficiently apply the greedy algorithm thanks to the following inequality \cite{negri2013reduced}:
\begin{equation*}
\norm{e}_{Y \times Y} \leq \frac{\norm{\Cal R}_{(Y \times Y)\dual}}{\beta_s^{LB}(\bmu)} \spazio \forall \bmu \in \Cal P.
\end{equation*}
{As a special case, one further} original contribution of this work is to {derive a novel} lower bound $\beta_s^{LB}(\bmu)$, {that allows to avoid having to rely on algorithms like the successive constraint methods} \cite{huynh2010natural}. Indeed, we obtain a lower bound for $\beta_{\Cal B_s^{N_h}} (\bmu)$ of the form:
\begin{equation}
\label{bounds}
\beta_s^{LB}(\bmu) = \frac{\min \{\alpha \gamma_a(\bmu), \gamma_a(\bmu)\}}{\sqrt{\max\{\alpha^2,1\}}}.
\end{equation}
as a consequence of Lemma \ref{lemma_surj_s} and Theorem \ref{s_FE}.
\end{remark}
In the next section, we will test the proposed error estimators both for time dependent and steady \ocp s governed by Graetz flows with $\Omega_u = \Omega_{\text{obs}}$.

%% file: Results.tex
\section{Numerical Results}
\label{results}
In this section, we will validate our estimator for \ocp s in two different parametrized settings: one dealing only with physical parameters, the other one with both physical and geometrical parametrization. The numerical tests, which are inspired by \cite{karcher2018certified,negri2013reduced,Strazzullo2}, are presented in their unsteady and steady versions. {The code used for obtaining these results is publicly available at \href{https://github.com/mstrazzu/parabolic_certified_RB_ocp/}{https://github.com/mstrazzu/parabolic$\_$certified$\_$RB$\_$ocp}. \\
We remark that these benchmarks are not meant to challenge the method, but to prove its reliability}. We specify that, in the numerical experiments, the spatial horizontal and vertical coordinates will be indicated with $x_1$ and $x_2$, respectively.

\subsection{Physical Parametrization}
This section deals with an OCP($\bmu$) governed by a Graetz flow with physical parametrization. We will solve the problem in $\Omega$, represented in Figure \eqref{dominio_1}, where the observation domain is $\Omega_{\text{obs}} = \Omega_1 \cup \Omega_2$, with $\Omega_1 = [0.2, 0.8]\times [0.3, 0.7]$ and $\Omega_2 = [1.2, 2.5]\times [0.3, 0.7]$. {To test the results of Section \ref{FEM}, we consider $\Omega_u = \Omega_{\text{obs}}$}.
\begin{figure}[H]

\begin{center}
\begin{tikzpicture}

\draw (0,0) -- (7.5,0) -- (7.5,3) -- (0,3) -- (0,0);
\filldraw[color=gray!60, fill=gray!10, thick](0.6,0.9) -- (2.4, 0.9) -- (2.4, 2.1) -- (0.6, 2.1) -- (0.6,0.9);
\filldraw[color=gray!60, fill=gray!10, thick](3.6,0.9) -- (7.49, 0.9) -- (7.49, 2.1) -- (3.6, 2.1) -- (3.6,0.9);
\filldraw[color=blue!60, fill=gray!10, very thick](0,0) -- (3.,0);
\filldraw[color=blue!60, fill=gray!10, very thick](0,0) -- (0.,3);
\filldraw[color=blue!60, fill=gray!10, very thick](0,3) -- (3.,3);
\filldraw[color=red!60, fill=gray!10, very thick, dashed](3,3) -- (7.5,3);
\filldraw[color=red!60, fill=gray!10, very thick, dashed](3,0) -- (7.5,0);

\node at (1.5,1.5){$\Omega_1$};
\node at (5.5,1.5){$\Omega_2$};
\node at (-.5,1.5){\color{blue}{$\Gamma_{D_1}$}};
\node at (5.5,3.3){\color{red}{$\Gamma_{D_2}$}};
\node at (7.8,1.5){\color{black}{$\Gamma_{N}$}};

\node at (0,-.3){\color{black}{$(0,0)$}};
\node at (3,-.3){\color{black}{$(1,0)$}};
\node at (7.5,-.3){\color{black}{$(2.5,0)$}};

\node at (0,3.3){\color{black}{$(0,1)$}};
\node at (3,3.3){\color{black}{$(1,1)$}};
\node at (7.5,3.3){\color{black}{$(2.5,1)$}};

\end{tikzpicture}
\end{center}
\caption{Domain $\Omega$. \textit{Observation {and control} domain:} $\Omega_u = \Omega_{\text{obs}} = \Omega_1 \cup \Omega_2$. \textit{Blue solid line:} first Dirichlet boundary conditions. \textit{Red dashed line:} second Dirichlet boundary conditions.}
\label{dominio_1}
\end{figure}
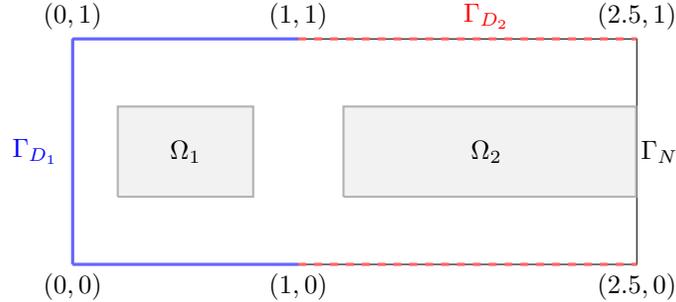
 The parameter is $\bmu \eqdot (\mu_1, \mu_2, \mu_3) \in \Cal P = [3,20]\times [0.5,1.5]\times[1.5,2.5]$, where $\mu_1$ represents the P\'eclet number of the \ocp $\;$ governed by advection-diffusion equation, while $\mu_2$ and $\mu_3$ are the constant desired state $y_d$ in the subdomains $\Omega_1$ and $\Omega_2$, respectively. Namely, we want to change a heat source to reach the parametrized observation. Thus, we solve the following problem: given $\bmu \in \Cal P$, find the pair $(y,p) \in \Cal Y \times\Cal Y$ such that
\begin{equation}
\label{test_1_time}
\tag{$P$}
\begin{cases}
\displaystyle y(\chi_{\Omega_1} + \chi_{\Omega_2})  - \dt{p} - \frac{1}{\mu_1}\Delta p - x_2(1 -x_2)\frac{\partial p}{\partial x_1} =
\mu_2 \chi_{\Omega_1} + \mu_3 \chi_{\Omega_2} & \text{ in } \Omega \times (0,T), \\
\displaystyle  \dt{y} - \frac{1}{\mu_1}\Delta y - x_2(1 -x_2)\frac{\partial y}{\partial x_1} - \frac{1}{\alpha}p(\chi_{\Omega_1} + \chi_{\Omega_2}) =  0& \text{ in } \Omega \times (0,T), \\
y(0) = y_0  & \text{ in } \Omega , \\
p(T) = 0 & \text{ in } \Omega, \\
\displaystyle \frac{1}{\mu_1}\frac{\partial y}{\partial n} = 0 & \text{ on $\Gamma_N \times (0,T),$} \\
\displaystyle \frac{1}{\mu_1}\frac{\partial p}{\partial n} + x_2(1 -x_2)n_1p= 0 & \text{ on $\Gamma_N \times (0,T),$} \\
y = 1 \text{ and } p = 0 & \text{ on $\Gamma_{D_1} \times (0,T),$} \\
y = 2 \text{ and } p = 0 & \text{ on $\Gamma_{D_2} \times (0,T),$}
\end{cases}
\end{equation}
where {$y_0 = 0$},
$\Gamma_{D_1} = \partial \Omega \cap \{(x_1, x_2) \; |\; x_1 \leq 1\}$ and $\Gamma_{D_2} = \partial \Omega \cap \{(x_1, x_2) \; |\; 1 < x_1  < 2.5 \}$. Calling $\Gamma_D \eqdot \Gamma_{D_1} \cup \Gamma_{D_2}$, then $\Gamma_N = \partial \Omega \setminus \Gamma_D$, where $\Gamma_N$ Neumann boundary conditions have been applied, with $n = (n_1, n_2)$ normal outer vector with respect to the boundary.
Problem \eqref{test_1_time}, can be considered also in its steady version: given $\bmu \in \Cal P$, find the pair $(y,p) \in Y \times Y$ such that
\begin{equation}
\label{test_1}
\tag{$P_s$}
\begin{cases}
\displaystyle y(\chi_{\Omega_1} + \chi_{\Omega_2})  - \frac{1}{\mu_1}\Delta p - x_2(1 -x_2)\frac{\partial p}{\partial x_1} =
\mu_2 \chi_{\Omega_1} + \mu_3 \chi_{\Omega_2} & \text{ in } \Omega, \\
\displaystyle  - \frac{1}{\mu_1}\Delta y + x_2(1 -x_2)\frac{\partial y}{\partial x_1} - \frac{1}{\alpha}p(\chi_{\Omega_1} + \chi_{\Omega_2}) =  0& \text{ in } \Omega, \vspace{2mm}\\
\displaystyle \frac{1}{\mu_1}\frac{\partial y}{\partial n} = 0 & \text{ on $\Gamma_N,$} \\
\displaystyle \frac{1}{\mu_1}\frac{\partial p}{\partial n} + x_2(1 -x_2)n_1p= 0 & \text{ on $\Gamma_N,$} \\
y = 1 \text{ and } p = 0 & \text{ on $\Gamma_{D_1},$} \\
y = 2 \text{ and } p = 0 & \text{ on $\Gamma_{D_2}.$}
\end{cases}
\end{equation}
It is straightforward to notice that the system is affine decomposed. {Considering $\alpha = 0.01$,
we follow the bound in \eqref{bounds} and we consider
\begin{equation}
\label{quatities}
\beta^{LB}(\bmu) = \beta_s^{LB}(\bmu) = \alpha \gamma_a(\bmu).
\end{equation}}
We now explicit all the constants involved in the bound, which is valid both for the unsteady and the steady case. First of all, we recall that, as reported in \cite{quarteroni2008numerical}, for the problem at hand
\begin{equation}
\gamma_a(\bmu) \eqdot \frac{1}{\mu_1(1 + C_{\Omega}^2)},
\end{equation}
where $C_{\Omega}$ is the Poincar\'e constant which verifies $\norm{v}_{L^2(\Omega)} \leq C_{\Omega}\norm{v}_{H^1(\Omega)}$, for all $v \in H^1(\Omega)$.  The value of $C_{\Omega}$ has been computed by solving the related eigenvalue problem, which is still computationally feasible since it is time and parameter-independent and must be evaluated only once. \\
For both the test cases, $\Cal P_h$ is given by $N_{\text{max}} = 225$ parameters uniformly distributed in $\Cal P$ and we consider $\tau = 1\cdot 10^{-4}$ as a tolerance for the greedy algorithm. To evaluate the performance of the proposed strategy, we show an average error analysis and average effectivity analysis over $100$ parameters with a uniform distribution in $\Cal P$. Furthermore, we recall that to {spatially discretize} the pair $(y,p)$, we used $\mathbb P^1-\mathbb P^1$ elements, while for time discretization we chose $T = 5$ and $\Delta t = 1/4$, which leads to $N_t = 20$.
{
\begin{remark}
\label{rem:new}
 {For $N_t=20$,} we were not able to {estimate} the Babu\v ska inf-sup constant {as the minimum eigenvalue of a generalized eigenvalue problem applied to the whole optimality system \cite{MALKUS19811299}. Indeed, the computational costs related to these eigenvalue analyzes were prohibitive and did not allow for a} comparison of plots between the lower bound value and Babu\v ska inf-sup constant, as depicted in Figures \ref{babu_vs_lb}, \ref{babu_vs_lb_2_1}, and \ref{babu_vs_lb_2_2}, {which were obtained for $N_t=5$}.{We recall that in this contribution, we provide an error estimator that
does not require the computation of the Babu\v ska inf-sup constant. As a consequence, the
proposed certification strategy is more suitable for finer space and time discretizations with respect to the state-of-the-art approaches based on the approximation of the inf-sup constant.} 
 We also tried $\Delta t = 1/6$ and $N_t = 30$, making the tests comparable with \cite{Strazzullo2}, however, the results obtained by exploiting \eqref{quatities} for $N_t = 30$ are very similar to the case of $N_t=20$. We do not show them for the sake of brevity.
\end{remark}}
 {
\begin{remark}
\label{rem:hyp}
We stress that all the hypotheses needed to apply our findings are verified by this distributed control. Indeed, in our case:
\begin{itemize}
\item[(a)]  $\displaystyle m(y, z; \boldsymbol \mu) =   \int_{\Omega_1 \cup \Omega_2} yz d\Omega$. Namely, it is continuous and positive definite by $L^2(\Omega_{\text{obs}})-$scalar product.
\item[(b)] $\displaystyle n(u, v; \boldsymbol \mu) =  \int_{\Omega_u} uv d\Omega$ is the $L^2(\Omega)$-scalar product, where $U=L^2(\Omega)$.
\item[(c)] $a(y, z; \boldsymbol \mu) = \displaystyle \int_{\Omega} \left (\frac{1}{\mu_1} \nabla y \nabla z + x_2(1 -x_2)\frac{\partial y}{\partial x_1}z \right ) d\Omega$, that is coercive and continuous in the $Y-$norm \cite{quarteroni2008numerical}.
\item[(d)] $\displaystyle c(p, z; \boldsymbol \mu) =  \int_{\Omega_u} pz d\Omega$, thus is trivially continuous.
\end{itemize}
From  \eqref{test_1_time}, we can deduce that the problem is affinely decomposed (see \cite{tesi} for more details). The smoothness of the solution with respect to the parameter is related to the regularity of the weak optimality system: indeed, $\Cal B\disc((y\disc, p\disc), (z\disc, q\disc); \bmu)$ is inf-sup stable (thanks to the previous assumptions) and $C^{\infty}$ with respect to $\boldsymbol \mu$ (by the affine decomposition definition) while $\Cal F(\bmu) = \mu_2 \chi_{\Omega_1} + \mu_3 \chi_{\Omega_2}$ is continuous and $C^{\infty}$ with respect to $\boldsymbol \mu$. Using \cite[Proposition 5.2]{Quarteroni20151}, we can state the variation of the solution with respect to the parameter is $C^{\infty}$ with respect to $\boldsymbol \mu$.
This remark also holds for the steady problem \eqref{test_1} with the needed simplifications.
\end{remark}}

\subsubsection{Unsteady Case}
\label{sec:time_no_geo}
This Section tests the performance of the employment of the lower bound $\beta^{LB}(\bmu)$ to problem \eqref{test_1_time}, with fixed $\alpha = 0.01$. 
The exploited greedy algorithm reached the chosen tolerance {$\tau = 1 \cdot 10^{-4}$} with $N = 15$ which, with the aggregated spaces technique, leads to a reduced space of dimension $4N = 60$: much less compared to the high-fidelity one, with $\Cal N = N_h \cdot N_t = 181440$ {($N_h = 9072$ and $N_t=20$)}. This high difference in the dimensionality of the system, allows us to reach huge \emph{speed up}, i.e., the number of reduced simulations that can be performed in a single realization of the high-fidelity problem. In this case, averaging over $100$ parameters, we reach values {of speed up} around $1.3\cdot 10^4$ {(a FOM simulation takes around $50$ seconds while a ROM simulation takes around $3.8\cdot 10^{-3}$ seconds)}. {In Table \ref{tab:times1}, we report the global offline time of the procedure, together with some specifics about the snapshots computation, the computation of the lower bound, which do not depend on $N$, the time needed for the Riesz representation of the residual and the basis functions update. Despite the large computational effort required at the offline level, we can deduce that the RB problem pays off if at least seventy simulation are performed. Indeed, seventy FOM simulations takes around $3500$ seconds, while the offline procedure and the simulation of seventy RB solutions takes around $3448.74$ seconds. We recall that thanks to the lower bound no eigenvalues problem are solved and this lighten the total offline costs.}

\begin{table}[H]
\caption{\centering Unsteady case ($N_t=5$): performance analysis for the problem \eqref{test_1_time} for $\alpha = 0.01$. Average error, estimators and effectivity {$\eta =\Delta_N(\bmu)/\norm{e}_{\Cal Q\times \Cal Q}$} exploiting the lower bound $\beta^{LB}(\bmu)$ and the Babu\v ska inf-sup constant $\beta_{\Cal B}^{\Cal N}(\bmu)$, with respect to $N$.}
\label{tab:test_1_time}
\footnotesize{\begin{center}
\begin{tabular}{c|c|c|c|c|c|c|c|c}
\hline
\multirow{2}{*}{$N$} & \multicolumn{4}{c|}{$\beta^{LB}(\bmu)$}      & \multicolumn{4}{c}{$\beta_{\Cal B}^{\Cal N}(\bmu)$}                         \\ \cline{2-9}
                     & $\norm{e}_{\text{rel}}$ &$\norm{e}_{\text{abs}}$ & $\Delta_N(\bmu)$    & $\eta$            & $\norm{e}_{\text{rel}}$ &$\norm{e}_{\text{abs}}$ & $\Delta_N(\bmu)$    & $\eta$   \\ \hline
1                    & 6.63e--1&4.37e+0     & 1.99e+2   & 1.23e+2 & 7.25e--1 & 1.716e+1 & 1.02e+1 & 5.91e+0 \\ \hline
5                    & 5.08e--2&5.84e--1    & 1.62e+1    & 1.93e+2 & 4.16e--2& 6.67e--1& 5.82e--1& 7.67e+0\\ \hline
10                    & 1.73e--3&1.58e--1   & 3.35e-1   & 1.91e+2 & 2.03e--3& 1.88e--3& 1.77e--2& 8.03e+0\\ \hline
15                   & 1.03e--4&4.98e--4  & 2.07e--2   & 1.92e+1 & 9.91e--5& 1.19e--5& 9.53e--4&8.46e+0\\ \hline
\end{tabular}\end{center}
}
\end{table}

\begin{table}[H]
\caption{Unsteady case ($N_t=20$): performance analysis for the problem \eqref{test_1_time} for $\alpha = 0.01$. Average error, estimators and effectivity {$\eta =\Delta_N(\bmu)/\norm{e}_{\Cal Q\times \Cal Q}$} exploiting the lower bound $\beta^{LB}(\bmu)$ with respect to $N$.}
\label{tab:test_1_time_20}
\footnotesize{\begin{center}
\begin{tabular}{c|c|c|c|c}
\hline
\multirow{2}{*}{$N$} & \multicolumn{4}{c}{$\beta^{LB}$}                         \\ \cline{2-5}
                     & $\norm{e}_{\text{rel}}$ &$\norm{e}_{\text{abs}}$ & $\Delta_N(\bmu)$    & $\eta$  \\ \hline
1                   & 6.53e--1&3.72e+0     & 1.33e+2   & 3.57e+1  \\ \hline
5                    & 4.63e--2&1.96e--1    & 1.13e+1    & 5.56e+1 \\ \hline
10                     & 2.51e--3&6.18e--2   & 3.37e-1   & 6.01e+1 \\ \hline
15                   & 1.36e--4&3.65e--4  & 2.23e--2   & 6.12e+1 \\ \hline
\end{tabular}\end{center}
}
\end{table}

\begin{table}[H]
\caption{{Unsteady case ($N_t=20$): offline costs for the problem \eqref{test_1_time} for $\alpha = 0.01$.}}
\label{tab:times1}
\begin{center}
    
\footnotesize{\color{black}
\begin{tabular}{c|cccc}
\hline
\multicolumn{1}{c|}{\multirow{2}{*}{$N$}} & \multicolumn{4}{c}{Total offline cost $\sim$ 3448.84}                                   \\ \cline{2-5} 
\multicolumn{1}{c|}{} & \multicolumn{1}{c|}{Snapshots computation} & \multicolumn{1}{c|}{$\beta^{LB}$} & \multicolumn{1}{c|}{Rietz representation} & Basis update \\ \hline
1                                       & \multicolumn{1}{c|}{49.58s} & \multicolumn{1}{c|}{0.09s} & \multicolumn{1}{c|}{41.63s} & 3.84s \\ \hline
5                                       & \multicolumn{1}{c|}{49.88s} & \multicolumn{1}{c|}{0.09s} & \multicolumn{1}{c|}{118.21s} &  16.28s\\ \hline
10                                      & \multicolumn{1}{c|}{49.87s} & \multicolumn{1}{c|}{0.09s} & \multicolumn{1}{c|}{201.43s} & 27.31s  \\ \hline
15                                       & \multicolumn{1}{c|}{49.80s} & \multicolumn{1}{c|}{0.09s} & \multicolumn{1}{c|}{267.84} & 38.25 \\ \hline
\end{tabular}}
\end{center}

\end{table}

In Figure \ref{no_geo} we show some representative solutions for $t = 1, 3, 4.5$, with $\alpha = 0.01$ and $\bmu = (12.0, 1.0, 2.5)$. 
We decided to focus only on the state and the adjoint variables, since the control can be post-processed. The reduced model is able to reproduce the high-fidelity solution for all the time instances considered. Furthermore, we propose a performance comparison with respect to the use of the Babu\v ska inf-sup constant. Table \ref{tab:test_1_time} presents the average absolute and relative error
\begin{equation}
\label{eq:errors}
\norm{e}_{\text{abs}} \eqdot \sqrt{(
\norm{y^{\Cal N} - y_N}_{\mathcal Q}^2 + \norm{p^{\Cal N} - p_N}_{\mathcal Q}^2)} \quad \text{and}
\quad
\norm{e}_{\text{rel}} \eqdot \frac{\sqrt{(\norm{y^{\Cal N} - y_N}_{\mathcal Q}^2 + \norm{p^{\Cal N} - p_N}_{\mathcal Q}^2)}}{\sqrt{(\norm{y^{\Cal N}}_{\mathcal Q}^2 + \norm{p^{\Cal N}}_{\mathcal Q}^2)}},
\end{equation}

\begin{figure}[H]
\centering
\begin{subfigure}[b]{0.24\textwidth}
\centering
\includegraphics[width=\textwidth]{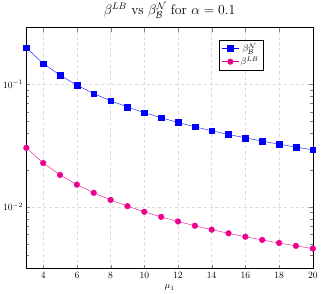}
\caption{\centering }
\label{fig:time01}
\end{subfigure}
\hfill
\begin{subfigure}[b]{0.24\textwidth}
\centering
\includegraphics[width=\textwidth]{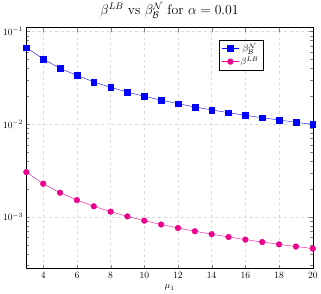}
\caption{\centering }
\label{fig:time001}
\end{subfigure}
\hfill
\begin{subfigure}[b]{0.24\textwidth}
\centering
\includegraphics[width=\textwidth]{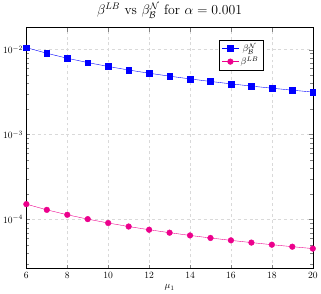}
\caption{\centering }
\label{fig:time0001}
\end{subfigure}
\caption{Comparison of the value of the lower bound $\beta^{LB}(\bmu)$ with respect to the exact Babu\v ska inf-sup constant $\beta_{\Cal B}^{\Cal N}(\bmu)$ for {$\alpha = 0.1,0.01, 0.001$} in (a),(b) and (c), respectively. The analysis has been performed by varying the value of $\mu_1$.}
\label{babu_vs_lb}
\end{figure}

together with the effectivity $\eta \eqdot \Delta_N(\bmu)/\norm{e}_{\Cal Q\times \Cal Q}$ and the estimator  value\footnote{\label{error_estimator}With abuse of notation we use $\Delta_N(\bmu)$ to describe both the \emph{exact error estimator}, given by the use of the Babu\v ska inf-sup constant $\beta_{\Cal B}^\Cal N(\bmu)$, and the \emph{surrogate error estimator} of $\Delta_N(\bmu)$, derived by employing the lower bound $\beta^{LB}(\bmu).$}, for $N_t=5$. The lower $N_t$ value is related to the computational costs related to the Babu\v ska inf-sup constant: this was the maximum value we could deal with to provide a complete performance analysis with our computational resources.  By definition, the lower bound $\beta^{LB}(\bmu)$ cannot give better results with respect to the exact computation of $\beta_{\Cal B}^{\disc}(\bmu)$. Still, exploiting the lower bound is of great convenience in the offline phase: indeed, the exact computation of the Babu\v ska inf-sup constant for a given $\bmu$ takes, on average, $58$ minutes, while the computation of the value of the lower bound is performed in $0.09$ seconds. Furthermore, the use of the lower bound very mildly affects the accuracy, since the errors are comparable between the two considered options.\\
In Table \ref{tab:test_1_time_20}, we propose the same performance analysis only for the lower bound case for $N_t=20$. To check the bound certification, in Figure \ref{babu_vs_lb}, we plot the values of the lower bound and the Babu\v ska inf-sup constant with respect $\mu_1$, which is the only parameter affecting the left-hand side of the system\footnote{In this case we were able to compute the approximation of the Babu\v ska inf-sup constant since we solved the problem only for 17 parameters for each value of $\alpha$.}. We stress that the eigenvalue solver struggled to approximate the Babu\v ska inf-sup constant for $\alpha =1$, for this reason, we only plot the results for $\alpha = 0.1, 0.01, 0.001$. We have to remark that the value of the penalization parameter $\alpha$ changes the tightness of the lower bound: for a lower $\alpha$, we have a worse approximation of the Babu\v ska inf-sup constant. This phenomenon is not new in literature, see, for example, \cite{karcher2018certified}. We can observe a good approximation of $\beta_{\Cal B}\disc(\bmu)$ in Figure \ref{fig:time01}, while we lose precision for smaller values of the penalization parameter as depicted in Figures \ref{fig:time001} and \ref{fig:time0001}.
\begin{figure}[H]
\centering
\begin{subfigure}[b]{0.3\textwidth}
\centering
\includegraphics[width=\textwidth]{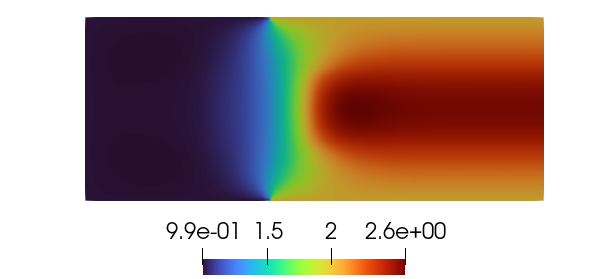}
\caption{\centering HF State Variable $t=1$}
\label{fig:off_p_1}
\end{subfigure}
\begin{subfigure}[b]{0.3\textwidth}
\centering
\includegraphics[width=\textwidth]{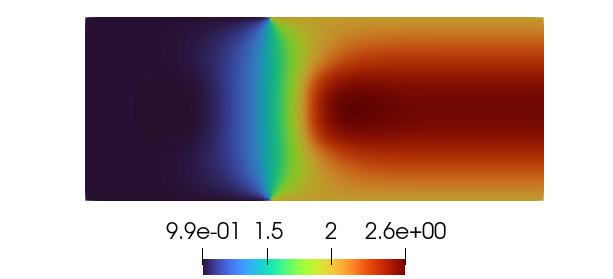}
\caption{\centering HF State Variable $t=3$}
\label{fig:off_p_3}
\end{subfigure}
\begin{subfigure}[b]{0.3\textwidth}
\centering
\includegraphics[width=\textwidth]{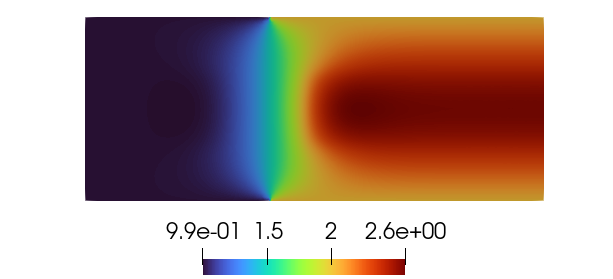}
\caption{\centering HF State Variable $t=4.5$}
\label{fig:off_p_5}
\end{subfigure}
\centering
\begin{subfigure}[b]{0.3\textwidth}
\centering
\includegraphics[width=\textwidth]{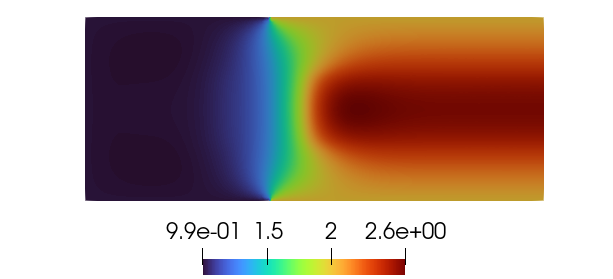}
\caption{\centering RB State Variable $t=1$}
\label{fig:on_y_1}
\end{subfigure}
\begin{subfigure}[b]{0.3\textwidth}
\centering
\includegraphics[width=\textwidth]{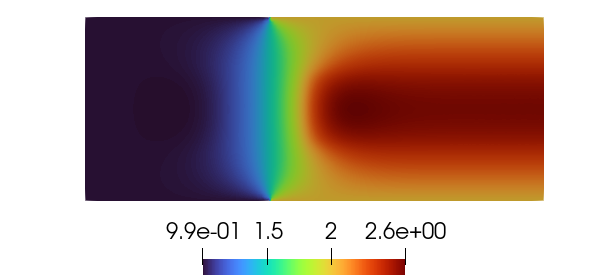}
\caption{\centering RB State Variable $t=3$}
\label{fig:on_p_3}
\end{subfigure}
\begin{subfigure}[b]{0.3\textwidth}
\centering
\includegraphics[width=\textwidth]{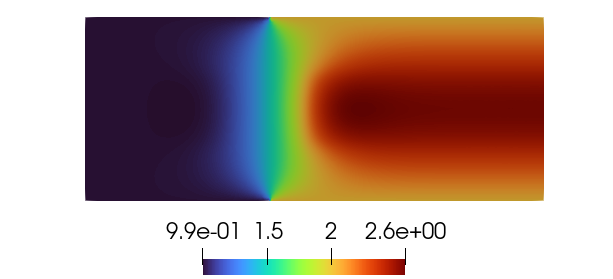}
\caption{\centering RB State Variable $t=4.5$}
\label{fig:on_p_5}
\end{subfigure}


\centering
\begin{subfigure}[b]{0.3\textwidth}
\centering
\includegraphics[width=\textwidth]{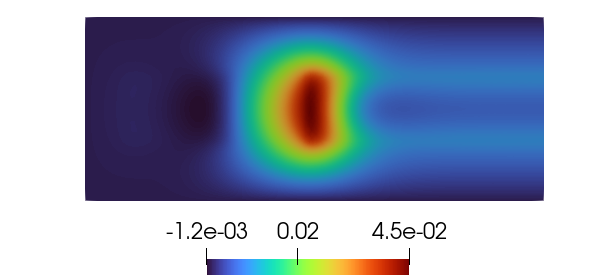}
\caption{\centering HF Adjoint Variable $t=1$}
\label{fig:off_p_1}
\end{subfigure}
\begin{subfigure}[b]{0.3\textwidth}
\centering
\includegraphics[width=\textwidth]{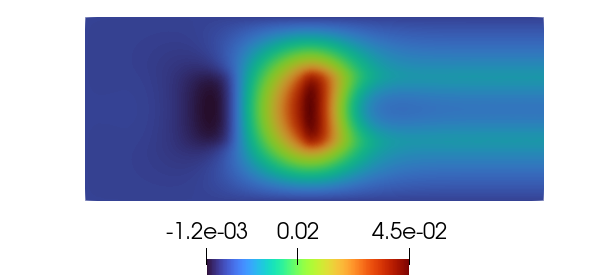}
\caption{\centering HF Adjoint Variable $t=3$}
\label{fig:off_p_3}
\end{subfigure}
\begin{subfigure}[b]{0.3\textwidth}
\centering
\includegraphics[width=\textwidth]{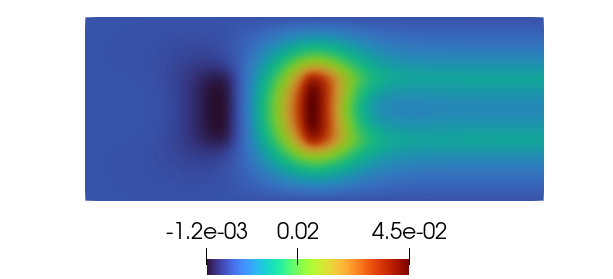}
\caption{\centering HF Adjoint Variable $t=4.5$}
\label{fig:off_p_5}
\end{subfigure}
\centering
\begin{subfigure}[b]{0.3\textwidth}
\centering
\includegraphics[width=\textwidth]{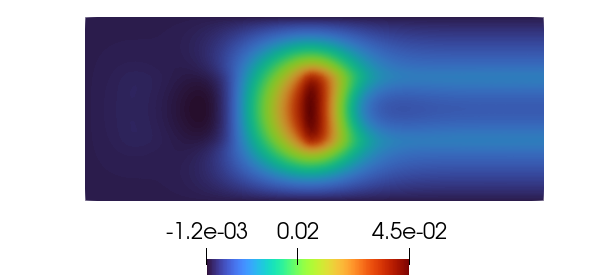}
\caption{\centering RB Adjoint Variable $t=1$}
\label{fig:on_p_1}
\end{subfigure}
\begin{subfigure}[b]{0.3\textwidth}
\centering
\includegraphics[width=\textwidth]{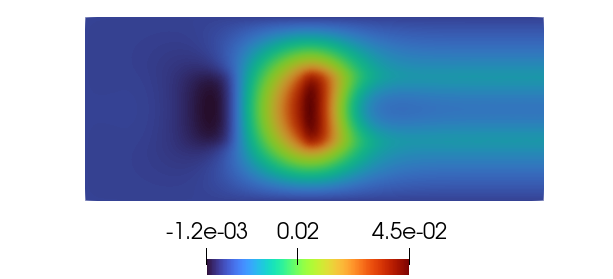}
\caption{\centering RB Adjoint Variable $t=3$}
\label{fig:on_p_3}
\end{subfigure}
\begin{subfigure}[b]{0.3\textwidth}
\centering
\includegraphics[width=\textwidth]{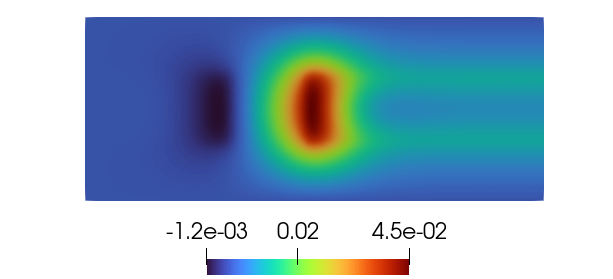}
\caption{\centering RB Adjoint Variable $t=4.5$}
\label{fig:on_p_5}
\end{subfigure}

\caption{Optimal HF and RB solutions with $\alpha = 0.01$ and $\bmu=(12.0, 1.0, 2.5)$. HF state variable for $t = 1, 3, 4.5$ in (a), (b), (c), respectively,  and RB state variable for $t = 1, 3, 4.5$ in (d), (e), (f). Analogously,  HF adjoint variable in {(g), (h), (i)}  and RB adjoint variable in {(j), (k), (l)} for $t = 1, 3, 4.5.$}
\label{no_geo}
\end{figure}

\subsubsection{Steady Case}
We report the results for the simulation of problem \eqref{test_1}. The high-fidelity problem has dimension $\Cal N = N_h = 9072$. We performed a greedy algorithm that reached the tolerance {$\tau = 1\cdot 10^{-4}$} after picking $N = 10$. In other words, the reduced system is much smaller with respect to the FE one, with a dimension of $4N = 40$, after the aggregated space procedure. The number $N$ is lower compared to the time dependent case: this is not surprising considering the simpler context of a steady problem. \\ We fixed $\alpha = 0.01$ and we propose a representative solution, obtained exploiting $\beta_s^{LB}$, in Figure \ref{s_no_geo} and an averaged performance analysis in Table \ref{tab:test_1}. The latter table shows the performance of the greedy considering approach with respect to an average absolute and relative error, given by
\begin{equation}
\label{eq:errors_s}
\norm{e}_{\text{abs}} \eqdot \sqrt{(\norm{y^{N_h} - y_N}_{Y}^2 + \norm{p^{N_h} - p_N}_{Y}^2)} \quad \text{and}
\quad
\norm{e}_{\text{rel}} \eqdot \frac{\sqrt{(\norm{y^{N_h} - y_N}_{Y}^2 + \norm{p^{N_h} - p_N}_{Y}^2)}}{\sqrt{(\norm{y^{N_h}}_{Y}^2 + \norm{p^{N_h}}_{Y}^2)}},
\end{equation}
respectively. The analysis has been carried out with the use of the lower bound and the Babu\v ska inf-sup constant and it also presents a comparison of the effectivity value $\eta \eqdot \Delta_N(\bmu)/\norm{e}_{Y\times Y}$ and the error estimator itself\footnote{See footnote \ref{error_estimator}.}. As already specified for the time dependent test case, the lower bound cannot perform better with respect to the exact Babu\v ska inf-sup constant. Indeed, the latter is preferable in terms of estimator and effectivity. Nonetheless, the lower bound gives comparable results in the average error analysis and we underline that exploiting $\beta^{LB}_s$ is computationally convenient also in the steady case. Indeed, the computation of $\beta_s^{LB}$ approximately takes 0.09s, while the eigenvalue problem associated with $\beta_{\Cal B_s}^{N_h}$ is solved in $0.38s$ and we recall that the process must be repeated for the $N_{\text{max}}$ parameters of $\Cal P_h$. In other words, employing the surrogate error estimator guarantees a faster offline phase, even for less complicated problems.
\begin{figure}
\centering
\begin{subfigure}[b]{0.3\textwidth}
\centering
\includegraphics[width=\textwidth]{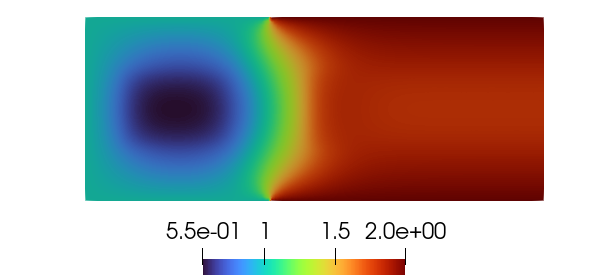}
\caption{\centering HF State Variable}
\label{fig:off_y}
\end{subfigure}
\begin{subfigure}[b]{0.3\textwidth}
\centering
\includegraphics[width=\textwidth]{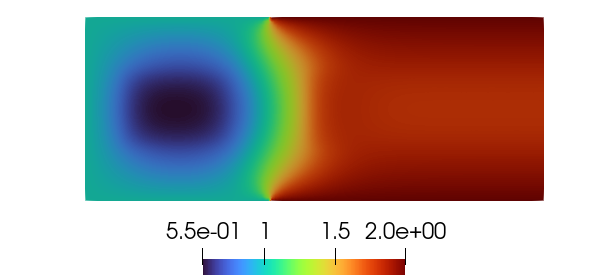}
\caption{\centering RB State Variable}
\label{fig:on_y}
\end{subfigure}\\
\begin{subfigure}[b]{0.3\textwidth}
\centering
\includegraphics[width=\textwidth]{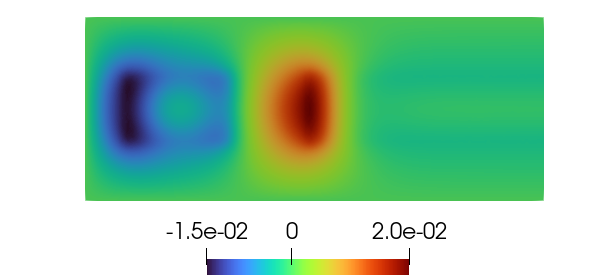}
\caption{\centering HF Adjoint Variable}
\label{fig:off_p}
\end{subfigure}
\begin{subfigure}[b]{0.3\textwidth}
\centering
\includegraphics[width=\textwidth]{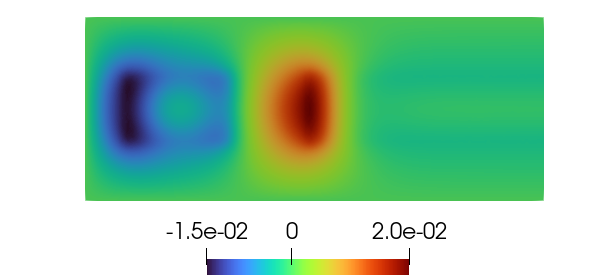}
\caption{\centering RB Adjoint Variable}
\label{fig:on_p}
\end{subfigure}
\hfill
\caption{Optimal HF and RB solutions with $\alpha = 0.01$ and $\bmu=(15.0, 0.6, 1.8)$. HF state variable in (a) and RB state variable in (b). Analogously, HF adjoint variable in (c) and RB adjoint variable in (d).}
\label{s_no_geo}
\end{figure}

 Furthermore, also in this case, the value of $\alpha$ influences the effectivity of the greedy procedure, as plots in Figure \ref{s_babu_vs_lb} show. If the value of $\alpha$ decreases, the effectivity increases, i.e.\ the lower bound becomes a worse approximation of the Babu\v ska inf-sup constant, as already observed in Section \ref{sec:time_no_geo}. A good approximation of $\beta_{\Cal B_s}^{N_h}(\bmu)$ is shown in Figure \ref{fig:babu_lb_1_steady}, while we lose precision for smaller values of the penalization parameter as depicted in Figures \ref{fig:babu_lb_2_steady}, \ref{fig:babu_lb_01_steady} and \ref{fig:babu_lb_001_steady}. As in the time dependent case, we showed the comparison varying the value of $\mu_1$. Indeed, the observation parameters $\mu_2$ and $\mu_3$, do not change the behavior of the Babu\v ska inf-sup since they only affect the right-hand side of the optimality system. The behavior of the constants is very similar to the unsteady ones. The main reason is in the space-time formulation itself, which presents almost the same structure of the simpler steady problem. \\
Finally, regarding the computational time gain of the reduced framework, we can also observe a quite important speed up in the steady case, which is around $150$, independently from the value of $N$. {We believe that larger values of $N$ would decrease the speedup, due to the dense nature of the reduced system.}

\begin{figure}
\centering
\begin{subfigure}[b]{0.24\textwidth}
\centering
\includegraphics[width=\textwidth]{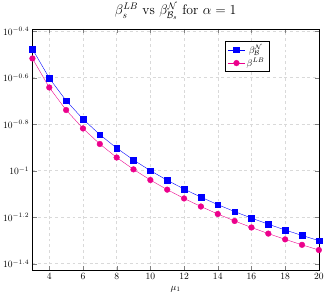}
\caption{\centering }
\label{fig:babu_lb_1_steady}
\end{subfigure}
\hfill
\begin{subfigure}[b]{0.24\textwidth}
\centering
\includegraphics[width=\textwidth]{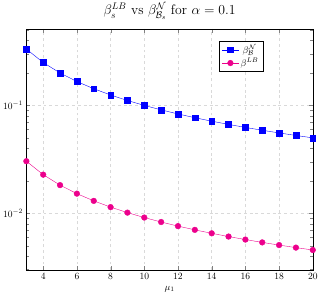}
\caption{\centering }
\label{fig:babu_lb_2_steady}
\end{subfigure}
\hfill
\centering
\begin{subfigure}[b]{0.24\textwidth}
\centering
\includegraphics[width=\textwidth]{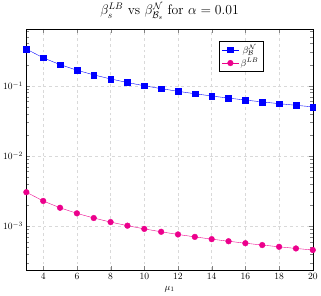}
\caption{\centering }
\label{fig:babu_lb_01_steady}
\end{subfigure}
\hfill
\begin{subfigure}[b]{0.24\textwidth}
\centering
\includegraphics[width=\textwidth]{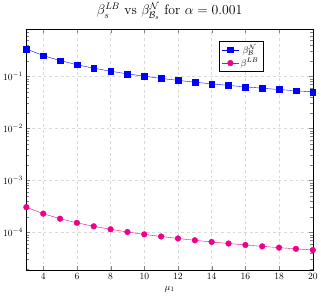}
\caption{\centering }
\label{fig:babu_lb_001_steady}
\end{subfigure}
\caption{Comparison of the value of the lower bound $\beta^{LB}_s(\bmu)$ with respect to the exact Babu\v ska inf-sup constant $\beta_{\Cal B_s}^{N_h}(\bmu)$ for $\alpha = 1, 0.1, 0.01, 0.001$ in (a), (b), (c) and (d), respectively. The analysis has been performed by varying the value of $\mu_1$.}
\label{s_babu_vs_lb}
\end{figure}

\begin{table}
\caption{ Steady case: performance analysis for the problem \eqref{test_1}. Average error, estimators and effectivity {$\eta =\Delta_N(\bmu)/\norm{e}_{Y \times Y}$} exploiting the lower bound $\beta_s^{LB}(\bmu)$ and the Babu\v ska inf-sup constant $\beta_s^{N_h}(\bmu)$, with respect to $N$.}
\label{tab:test_1}
\footnotesize{\begin{center}
\begin{tabular}{c|c|c|c|c|c|c|c|c}
\hline
\multirow{2}{*}{$N$} & \multicolumn{4}{c|}{$\beta^{LB}_s(\bmu)$}      & \multicolumn{4}{c}{$\beta_{\Cal B_s}^{N_h}(\bmu)$}                         \\ \cline{2-9}
                     & $\norm{e}_{\text{rel}}$ &$\norm{e}_{\text{abs}}$ & $\Delta_N(\bmu)$    & $\eta$            & $\norm{e}_{\text{rel}}$ &$\norm{e}_{\text{abs}}$ & $\Delta_N(\bmu)$    & $\eta$   \\ \hline
1                    & 7.76e--1&6.08e--1     & 6.85e+1   & 8.34e+1 & 7.16e--1 & 7.41e--1 & 1.55e+0 & 2.09e+0 \\ \hline
5                    & 2.00e--2&5.77e--2   & 2.28e+0   & 1.13e+2 & 3.13e--2& 1.49e--2& 9.50e--&2.72e+0\\ \hline
10                   & 2.53e--4&8.17e--4  &2.00e--2 & 1.06e+2 & 7.18e--4& 2.29e--4& 6.13e--4&2.65e+0\\ \hline
\end{tabular}\end{center}
}
\end{table}
In the next Section, we will show how the proposed lower bound performs in a more complex test case, which also considers geometrical parametrization.

\subsection{Physical and Geometrical Parametrization} In this Section, we propose a boundary \ocp $\;$ governed by a Graetz flow with physical and geometrical parametrization. As parameter dependent domain, we consider $\Omega(\bmu)$ depicted in Figure \ref{dominio_2}. In this case the observation domain is $\Omega_{\text{obs}} (\bmu)= \Omega_3(\bmu) \cup \Omega_4(\bmu)$, where $\Omega_1$ is the unit square, $\Omega_3(\bmu) = [1, 1 + \mu_2]\times [0.8, 1]$,  $\Omega_4(\bmu) = [1, 1 + \mu_2]\times [0., 0.2]$, while $\Omega_2(\bmu) = [1, 1 + \mu_2]\times [0.2, 0.8]$. The control acts on ${\Omega_u = \Omega_{\text{obs}}}$.

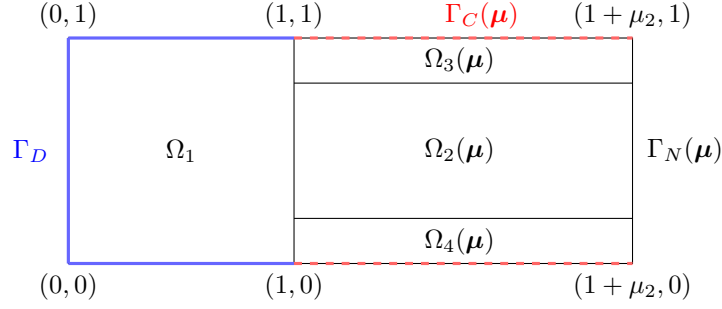
\begin{figure}[H]
\begin{center}
\begin{tikzpicture}

\draw (0,0) -- (7.5,0) -- (7.5,3) -- (0,3) -- (0,0);
\draw (3,0) -- (3, 3);
\draw (3,0.6) -- (7.5, 0.6);
\draw (3,2.4) -- (7.5, 2.4);

\filldraw[color=blue!60, fill=gray!10, very thick](0,0) -- (3.,0);
\filldraw[color=blue!60, fill=gray!10, very thick](0,0) -- (0.,3);
\filldraw[color=blue!60, fill=gray!10, very thick](0,3) -- (3.,3);
\filldraw[color=red!60, fill=gray!10, very thick, dashed](3,3) -- (7.5,3);
\filldraw[color=red!60, fill=gray!10, very thick, dashed](3,0) -- (7.5,0);

\node at (1.5,1.5){$\Omega_1$};
\node at (5.2,1.5){$\Omega_2(\bmu)$};
\node at (5.2,2.7){$\Omega_3(\bmu)$};
\node at (5.2,0.3){$\Omega_4(\bmu)$};

\node at (-.5,1.5){\color{blue}{$\Gamma_{{D}_{1}}$}};
\node at (5.5,3.3){\color{red}{$\Gamma_{{D}_{2}}(\bmu)$}};
\node at (8.2,1.5){\color{black}{$\Gamma_{N}(\bmu)$}};

\node at (0,-.3){\color{black}{$(0,0)$}};
\node at (3,-.3){\color{black}{$(1,0)$}};
\node at (7.5,-.3){\color{black}{$(1 + \mu_2,0)$}};

\node at (0,3.3){\color{black}{$(0,1)$}};
\node at (3,3.3){\color{black}{$(1,1)$}};
\node at (7.5,3.3){\color{black}{$(1 + \mu_2,1)$}};

\end{tikzpicture}
\end{center}
\caption{Domain $\Omega$. \textit{{Control and} Observation domain:} $\Omega_u = \Omega_{\text{obs}}(\bmu) = \Omega_3(\bmu)\cup \Omega_4(\bmu)$.
\textit{Blue solid line:} first Dirichlet boundary condition, \textit{Red dashed line:} second Dirichlet boundary condition,. The reference domain $\Omega$ is given by $\mu_2 = 1$.}
\label{dominio_2}
\end{figure}

 The parameter is $\bmu \eqdot (\mu_1, \mu_2, \mu_3) \in
\Cal P = [6.0, 20.0] \times [1.0, 3.0] \times [0.5, 3.0]$, where $\mu_1$, as the previous test case, represents the P\'eclet number, $\mu_2$ is a geometrical parameter which stretches the length of the right part of the domain, $\mu_3$ is a constant desired state $y_d$ observed in $\Omega_{\text{obs}}(\bmu)$. The problem reads as follows: given $\bmu \in \Cal P$, find the pair $(y,p) \in \Cal Y \times\Cal Y$ such that
\begin{equation}
\label{test_2_time}
\tag{$P^g$}
\begin{cases}
\displaystyle y\chi_{\Omega_{\text{obs}}(\bmu)}  - \dt{p} - \frac{1}{\mu_1}\Delta p - x_2(1 -x_2)\frac{\partial p}{\partial x_1} = 
\mu_3 \chi_{\Omega_{\text{obs}}(\bmu)}  & \text{ in } \Omega(\bmu) \times (0,T), \\
\displaystyle  \dt{y} - \frac{1}{\mu_1}\Delta y + x_2(1 -x_2)\frac{\partial y}{\partial x_1} - \frac{1}{\alpha} p \chi_{\Omega_{\text{obs}}(\bmu)}=  0& \text{ in } \Omega(\bmu) \times (0,T), \\
y(0) = y_0  & \text{ in } \Omega(\bmu) , \\
p(T) = 0 & \text{ in } \Omega(\bmu), \\
\displaystyle \frac{1}{\mu_1}\frac{\partial y}{\partial n} = 0 & \text{ on $\Gamma_N(\bmu) \times (0,T),$} \\
\displaystyle \frac{1}{\mu_1}\frac{\partial p}{\partial n} + x_2(1 -x_2)n_1p= 0 & \text{ on $\Gamma_N(\bmu) \times (0,T),$} \\
y = 1 \text{ and } p = 0 & \text{ on $\Gamma_{D_1} \times (0,T),$} \\
y = 0 \text{ and } p = 0 & \text{ on $\Gamma_{D_2}(\bmu) \times (0,T),$} \\
\end{cases}
\end{equation}
where {$y_0 = 0$} in $\Omega(\bmu)$,
$\Gamma_{D_1} = \partial \Omega(\bmu) \cap \{(x_1, x_2) \; |\; x_1 \leq 1\}$ and  $\Gamma_{D_2} = \partial \Omega \cap \{(x_1, x_2) \; |\; 1 < x_1  < 1 + \mu_2 \}$ are the Dirichlet boundary domain, while $\Gamma_N(\bmu) = \partial \Omega(\bmu) \setminus \{ \Gamma_{D_1} \cup \Gamma_{D_2}(\bmu) \}$, is the Neumann boundary. Moreover, $n = (n_1, n_2)$ is the outer normal vector to $\Gamma_N(\bmu)$.
Also for this case, we can focus our attention on the steady version, i.e.: given $\bmu \in \Cal P$, find the pair $(y,p) \in Y \times Y$ such that
\begin{equation}
\label{test_2}
\tag{$P_s^g$}
\begin{cases}
\displaystyle y {\chi_{\Omega_{\text{obs}}}}  - \frac{1}{\mu_1}\Delta p - x_2(1 -x_2)\frac{\partial p}{\partial x_1} = 
\mu_3 \chi_{\Omega_{\text{obs}}(\bmu)} & \text{ in } \Omega(\bmu), \\
\displaystyle  - \frac{1}{\mu_1}\Delta y + x_2(1 -x_2)\frac{\partial y}{\partial x_1} - \frac{1}{\alpha} p \chi_{\Omega_{\text{obs}}(\bmu)} =  0& \text{ in } \Omega(\bmu), \vspace{2mm}\\
\displaystyle \frac{1}{\mu_1}\frac{\partial y}{\partial n} = 0 & \text{ on $\Gamma_N(\bmu),$} \\
\displaystyle \frac{1}{\mu_1}\frac{\partial p}{\partial n} + x_2(1 -x_2)n_1p= 0 & \text{ on $\Gamma_N(\bmu),$} \\
y = 1 \text{ and } p = 0 & \text{ on $\Gamma_{D_1},$} \\
y = 1 \text{ and } p = 0 & \text{ on $\Gamma_{D_2}(\bmu),$} \\
\end{cases}
\end{equation}
Also for this test case, we exploit the lower bound \eqref{quatities} for $\alpha = 0.07$, both for the unsteady and the steady case. To try a different setting, we employ a different set of Dirichlet boundary conditions between the two cases.
Every step of the greedy algorithm is performed in the reference domain $\Omega$ given by $\mu_2 = 1$ after the application of a trace-back operator, the interested reader can find the details in \cite{tesi}. We underline that the choice of the reference parameter affects {Poincar\'e} constant. However, for the sake of notation, we will omit the parameter dependency.
In this case,
\begin{equation}
\gamma_a(\bmu) \eqdot \min \left \{ \frac{1}{\mu_1}, \frac{1}{\mu_1\mu_2}, \frac{\mu_2}{\mu_1}, 1\right \}\frac{1}{(1 + C_{\Omega}^2)},
\end{equation}
where $C_{\Omega}$ is the Poincar\'e associated to the reference domain. We remark that the Poincar\'e $C_{\Omega}$ can be evaluated directly through the resolution of an eigenvalue problem which has to be performed only once since defined in the reference domain $\Omega$. In other words, its exact computation does not impact the offline performance of the reduced space building procedure. For the space-time discretization we exploited $\mathbb P^1-\mathbb P^1$ elements and $N_t = 20$ in the time interval $[0, T] = {[0,5]}$, i.e. $\Delta t = 1/4$. 
Also for this test case, we perform a greedy algorithm on $\Cal P_h$ of cardinality $N_{\text{max}} = 225$ picked through a uniform distribution on $\Cal P$. The tolerance has been chosen as $\tau = 1\cdot 10^{-4}$. For the steady and the unsteady case, we propose a performance analysis over $100$ parameters with a uniform distribution in $\Cal P$. The results for the unsteady and the steady case follow.
\begin{remark}
Also for this experiment setting, we can apply the error estimator. Indeed, all the needed hypotheses are verified since we are in the same setting of Remark \ref{rem:hyp}.
\end{remark}

\subsubsection{Unsteady Case}
\label{sec:time_geo}
This section deals with problem \eqref{test_2_time}, with fixed $\alpha = 0.07$, to which we applied the greedy algorithm using the lower bound \eqref{quatities}. To achieve the tolerance $\tau$, the offline phase took $N = 26$ basis functions. Once again, exploiting the aggregated spaces strategy, we ended up with a reduced space of global dimension $4N = 104$. The problem at hand, due to the geometrical parametrization, needs a greater number of bases with respect to the previous example. Still, the global reduced dimension is convenient compared to the high-fidelity dimension $\Cal N = N_h \cdot N_t = 207320$ {($N_h = 10366$ and $N_t=20$)}. This is the reason we reach, also in this case, a good speed up around the value of $4 \cdot 10^5$ for all $N$, obtained averaging over $100$ parameters uniformly distributed in $\Cal P$. {(a FOM simulation takes around $60$ seconds while a ROM simulation takes around $1.5\cdot 10^{-3}$ seconds)}. {In Table \ref{tab:times}, we report the global offline time of the procedure with the specifics about the snapshots computation, the computation of the lower bound, the time needed for the Riesz representation of the residual and the basis functions update, varying $N$. This test case has a larger offline costs due to the increased complexity due to the geometric parametrization. From the table, we deduce that the RB problem pays off if at least around $263$ simulations are performed. Indeed, $263$ takes around $15780$ seconds at the FOM level, while the offline procedure and the simulation of $263$ RB solutions takes around $15734.35$ seconds.} Figure \ref{geo} shows representative HF and RB state and adjoint solutions for $t = 1, 3, 4.5$, with $\alpha = 0.07$ and $\bmu = (15.0, 1.5, 2.5)$. The RB solutions recover the HF behavior.

The performance of the greedy algorithm is also evaluated in comparison with Babu\v ska inf-sup for $N_t = 5$ in Table \ref{tab:test_2_time}. We recall that $N_t=5$ was the maximum number of time instances for which we were able to perform a complete performance analysis for the Babu\v ska inf-sup constant with our computational resources. The table represents the average errors defined in \eqref{eq:errors}
together with the effectivity and the error estimator\footnote{See footnote \ref{error_estimator}.} with respect to the Babu\v ska inf-sup constant $\beta_{\Cal B}^\Cal N(\bmu)$ and the lower bound. The greedy procedure stops at $N=22$ because we are considering fewer time instances. In this case, we lose in effectivity, but we recall that computing $225$ Babu\v ska inf-sup constants is much more expensive than evaluating the value of $\beta^{LB}(\bmu)$. Indeed, the computation of the first constant, for a single parameter in this case, takes around $52$ minutes to be computed, on average, while, as already said, $\beta^{LB}(\bmu)$ is computed in only $0.09$ seconds.
\begin{figure}
\centering
\begin{subfigure}[b]{0.3\textwidth}
\centering
\includegraphics[width=\textwidth]{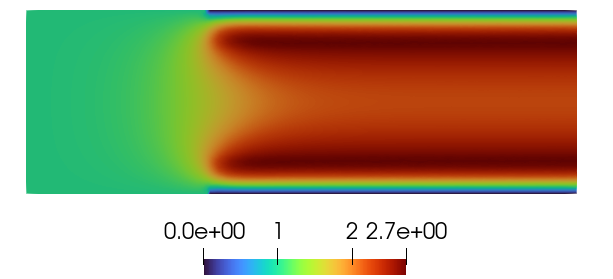}
\caption{\centering HF State Variable $t=1$}
\label{fig:off_p_1}
\end{subfigure}
\begin{subfigure}[b]{0.3\textwidth}
\centering
\includegraphics[width=\textwidth]{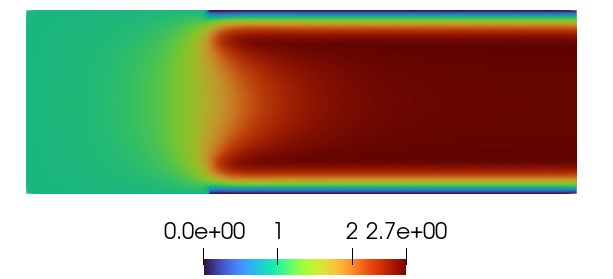}
\caption{\centering HF State Variable $t=3$}
\label{fig:off_p_3}
\end{subfigure}
\begin{subfigure}[b]{0.3\textwidth}
\centering
\includegraphics[width=\textwidth]{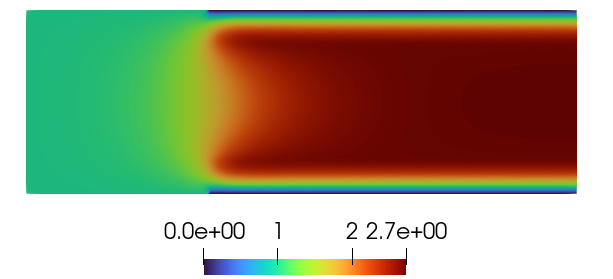}
\caption{\centering HF State Variable $t=4.5$}
\label{fig:off_p_5}
\end{subfigure}
\centering
\begin{subfigure}[b]{0.3\textwidth}
\centering
\includegraphics[width=\textwidth]{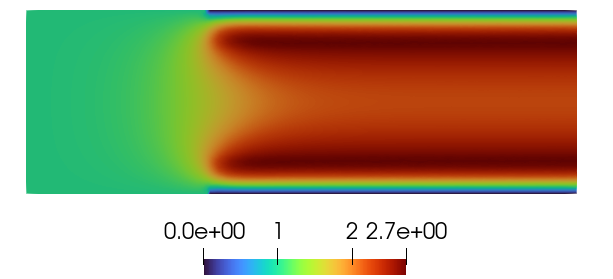}
\caption{\centering RB State Variable $t=1$}
\label{fig:on_y_1}
\end{subfigure}
\begin{subfigure}[b]{0.3\textwidth}
\centering
\includegraphics[width=\textwidth]{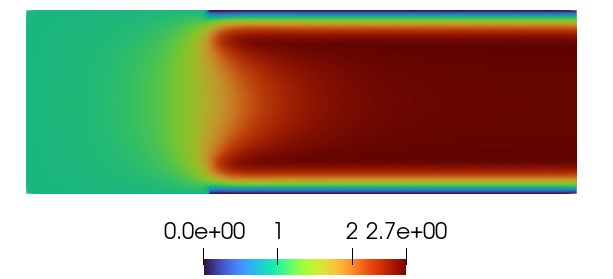}
\caption{\centering RB State Variable $t=3$}
\label{fig:on_p_3}
\end{subfigure}
\begin{subfigure}[b]{0.3\textwidth}
\centering
\includegraphics[width=\textwidth]{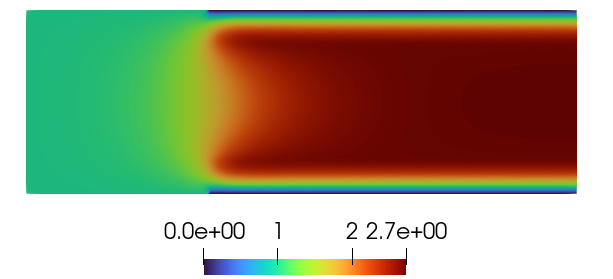}
\caption{\centering RB State Variable $t=4.5$}
\label{fig:on_p_5}
\end{subfigure}


\centering
\begin{subfigure}[b]{0.3\textwidth}
\centering
\includegraphics[width=\textwidth]{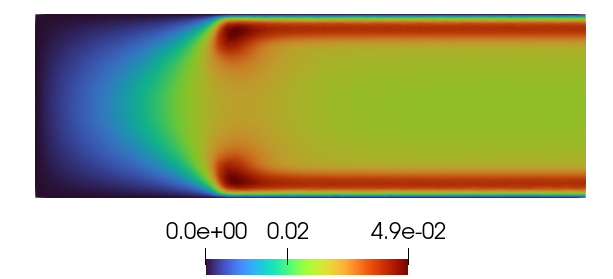}
\caption{\centering HF Adjoint Variable $t=1$}
\label{fig:off_p_1}
\end{subfigure}
\begin{subfigure}[b]{0.3\textwidth}
\centering
\includegraphics[width=\textwidth]{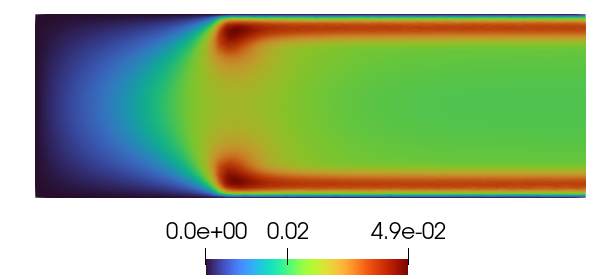}
\caption{\centering HF Adjoint Variable $t=3$}
\label{fig:off_p_3}
\end{subfigure}
\begin{subfigure}[b]{0.3\textwidth}
\centering
\includegraphics[width=\textwidth]{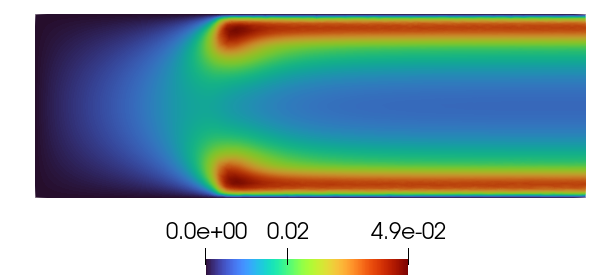}
\caption{\centering HF Adjoint Variable $t=4.5$}
\label{fig:off_p_5}
\end{subfigure}
\centering
\begin{subfigure}[b]{0.3\textwidth}
\centering
\includegraphics[width=\textwidth]{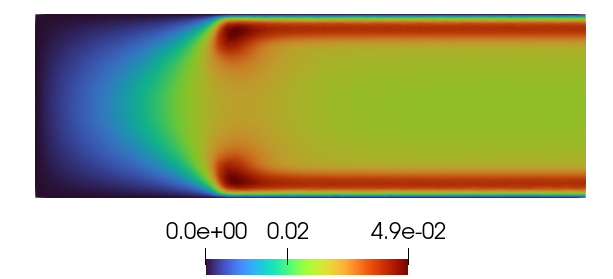}
\caption{\centering RB Adjoint Variable $t=1$}
\label{fig:on_p_1}
\end{subfigure}
\begin{subfigure}[b]{0.3\textwidth}
\centering
\includegraphics[width=\textwidth]{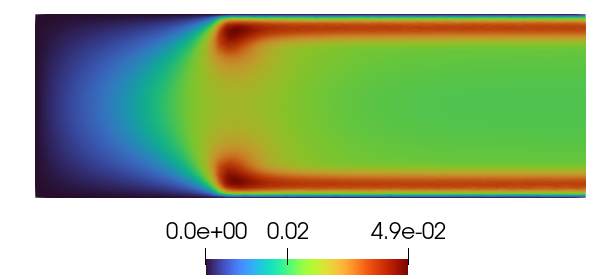}
\caption{\centering RB Adjoint Variable $t=3$}
\label{fig:on_p_3}
\end{subfigure}
\begin{subfigure}[b]{0.3\textwidth}
\centering
\includegraphics[width=\textwidth]{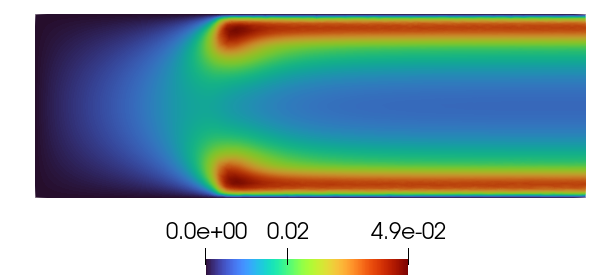}
\caption{\centering RB Adjoint Variable $t=4.5$}
\label{fig:on_p_5}
\end{subfigure}

\caption{Optimal HF and RB steady control variables with $\alpha = 0.07$ and $\bmu= (15.0, 1.5, 2.5)$. HF state variable for $t = 1, 3, 4.5$ in (a), (b), (c), respectively,  and the RB state variable for $t = 1, 3, 4.5$ in (d), (e), (f). Analogously,  HF adjoint variable in {(g), (h), (i)}  and RB adjoint variable in {(j), (k), (l)} for $t = 1, 3, 4.5$.}
\label{geo}
\end{figure}
Furthermore, the relative errors of the two approaches are comparable (first and fifth columns of Table \ref{tab:test_2_time}, in some cases the employment of the lower bound looks even more accurate, but it is due to different testing sets for the two experiments). Table \ref{tab:test_2_time_20} shows the performances of the proposed approach for $N_t=20$.
To check the properties of the lower bound for $N_t=20$, we compare it to the Babu\v ska inf-sup constant in Figures \ref{babu_vs_lb_2_1} and \ref{babu_vs_lb_2_2}. We compute the two constants for some values of $\mu_1$, for several values of $\alpha$, with $\mu_2 = 1$ and $\mu_2 = 2$, respectively. {We recall that $\mu_3$ does not affect the computations, since it is present as a forcing term in the equations.} In this case, the addition of the geometrical parametrization slightly influences the bound. As expected, the bound is sensitive to the value of $\alpha$. We stress that the plateau feature of plot \ref{fig:babu_lb_3_2} is due to the shift-parameter choice of the eigenvalue solver. With other shift-parameters, we expect a decreasing trend of $\beta \disc_{\mathcal B}(\bmu)$. However, finding the best parameter for the approximation of the  Babu\v ska inf-sup constant was beyond the goal of the paper. Rather, this shows how convenient the employment of the lower bound can be.

\begin{table}[H]
\caption{\centering Unsteady case ($N_t=5$): performance analysis for the problem \eqref{test_2_time} for $\alpha = 0.07$. Average error, estimators and effectivity {$\eta =\Delta_N(\bmu)/\norm{e}_{\Cal Q\times \Cal Q}$} exploiting the lower bound $\beta^{LB}(\bmu)$ and the Babu\v ska inf-sup constant $\beta_{\Cal B}^{\Cal N}(\bmu)$, with respect to $N$.}
\label{tab:test_2_time}
\footnotesize{\begin{center}
\begin{tabular}{c|c|c|c|c|c|c|c|c}
\hline
\multirow{2}{*}{$N$} & \multicolumn{4}{c|}{$\beta^{LB}(\bmu)$}      & \multicolumn{4}{c}{$\beta_{\Cal B}^{\Cal N}(\bmu)$}                         \\ \cline{2-9}
                     & $\norm{e}_{\text{rel}}$ &$\norm{e}_{\text{abs}}$ & $\Delta_N(\bmu)$    & $\eta$            & $\norm{e}_{\text{rel}}$ &$\norm{e}_{\text{abs}}$ & $\Delta_N(\bmu)$    & $\eta$   \\ \hline
1                    & 1.86e--1&1.36e+0     & 6.17e+2   & 4.54e+2 & 1.17e--1 & 1.29e+0 & 1.55e+1 & 1.19e+1 \\ \hline
6                    & 1.01e--2&6.63e--2    & 2.92e+1    & 1.93e+2 & 4.16e--2& 6.01e--2& 7.22e--1& 1.16e+1\\ \hline
12                    & 1.37e--3&1.16e--2   & 4.85e+0   & 4.40e+2 & 2.03e--3& 9.22e--3& 1.27e--1& 1.31e+1\\ \hline
18                   & 2.23e--4&1.17e--3  & 8.35e--1   & 4.17e+2 & 9.91e--5& 1.13e--3& 2.77e--2&1.06e+1\\ \hline
22                   & 9.23e--5&6.63e--4  & 3.33e--1   & 3.83e+2 & 9.91e--5& 6.16e--4& 3.10e--2&1.04e+1\\ \hline

\end{tabular}\end{center}
}
\end{table}

\begin{table}[H]
\caption{Unsteady case ($N_t=20$): performance analysis for the problem \eqref{test_2_time} for $\alpha = 0.07$. Average error, estimators and effectivity {$\eta =\Delta_N(\bmu)/\norm{e}_{\Cal Q\times \Cal Q}$} exploiting the lower bound $\beta^{LB}(\bmu)$ with respect to $N$.}
\label{tab:test_2_time_20}
\footnotesize{\begin{center}
\begin{tabular}{c|c|c|c|c}
\hline
\multirow{2}{*}{$N$} & \multicolumn{4}{c}{$\beta^{LB}$}                         \\ \cline{2-5}
                     & $\norm{e}_{\text{rel}}$ &$\norm{e}_{\text{abs}}$ & $\Delta_N(\bmu)$    & $\eta$  \\ \hline
1                   & 1.91e--1&3.46e+0     & 4.26e+2   & 1.23e+2  \\ \hline
8                    & 8.53e--3&1.38e--1    & 1.32e+1    & 9.86e+1 \\ \hline
16                     & 7.51e--4&1.25e--2   & 1.34e+0   & 1.07e+2 \\ \hline
24                   & 1.36e--4&2.07e--3  & 1.70e--2   & 1.09e+2 \\ 
\hline
26                   & 1.05e--4&1.74e--3  & 1.67e--2   & 9.68e+1 \\ 
\hline

\end{tabular}\end{center}
}
\end{table}

\begin{table}[]
\caption{{Unsteady case ($N_t=20$): offline costs for the problem \eqref{test_2_time} for $\alpha = 0.07$}.}
\label{tab:times}
\begin{center}
    
\footnotesize{\color{black}
\begin{tabular}{c|cccc}
\hline
\multicolumn{1}{c|}{\multirow{2}{*}{$N$}} & \multicolumn{4}{c}{Total offline cost $\sim$ 15733.96 s}                                   \\ \cline{2-5} 
\multicolumn{1}{c|}{} & \multicolumn{1}{c|}{Snapshots computation} & \multicolumn{1}{c|}{$\beta^{LB}$} & \multicolumn{1}{c|}{Rietz representation} & Basis update \\ \hline
1                                       & \multicolumn{1}{c|}{60.93s} & \multicolumn{1}{c|}{0.09s} & \multicolumn{1}{c|}{90.43s} & 3.73s \\ \hline
6                                       & \multicolumn{1}{c|}{61.09s} & \multicolumn{1}{c|}{0.09s} & \multicolumn{1}{c|}{283.31s} &  20.46s\\ \hline
12                                      & \multicolumn{1}{c|}{60.97s} & \multicolumn{1}{c|}{0.09s} & \multicolumn{1}{c|}{525.12s} & 37.62s  \\ \hline
18                                      & \multicolumn{1}{c|}{61.37s} & \multicolumn{1}{c|}{0.09s} & \multicolumn{1}{c|}{702.14s} & 50.58s  \\ \hline
22                                       & \multicolumn{1}{c|}{61.04s} & \multicolumn{1}{c|}{0.09s} & \multicolumn{1}{c|}{905.23s} & 61.99s \\ \hline

\end{tabular}}
\end{center}

\end{table}

\begin{figure}
\centering
\begin{subfigure}[b]{0.24\textwidth}
\centering
\includegraphics[width=\textwidth]{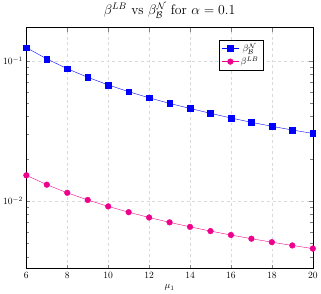}
\caption{\centering }
\label{fig:babu_lb_1}
\end{subfigure}
\hfill
\begin{subfigure}[b]{0.24\textwidth}
\centering
\includegraphics[width=\textwidth]{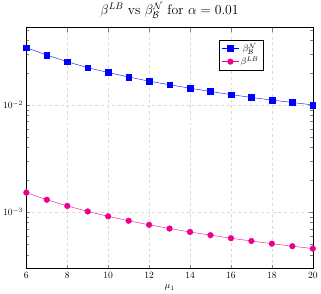}
\caption{\centering }
\label{fig:babu_lb_2}
\end{subfigure}
\hfill
\begin{subfigure}[b]{0.24\textwidth}
\centering
\includegraphics[width=\textwidth]{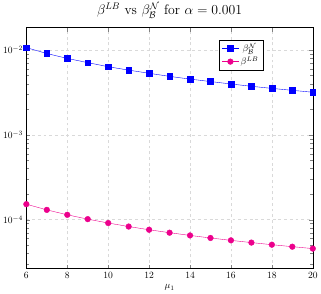}
\caption{\centering }
\label{fig:babu_lb_3}
\end{subfigure}
\caption{Comparison of the value of the lower bound $\beta^{LB}(\bmu)$ with respect to the exact Babu\v ska inf-sup constant $\beta_{\Cal B}^{\Cal N}(\bmu)$ for $\alpha = 0.1, 0.01, 0.001$ in (a), (b) and (c), respectively. The analysis has been performed varying the value of $\mu_1$ and fixing $\mu_2 = 1$.}
\label{babu_vs_lb_2_1}
\end{figure}

\begin{figure}
\centering
\begin{subfigure}[b]{0.24\textwidth}
\centering
\includegraphics[width=\textwidth]{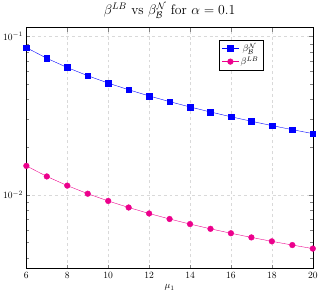}
\caption{\centering }
\label{fig:babu_lb_1_2}
\end{subfigure}
\hfill
\begin{subfigure}[b]{0.24\textwidth}
\centering
\includegraphics[width=\textwidth]{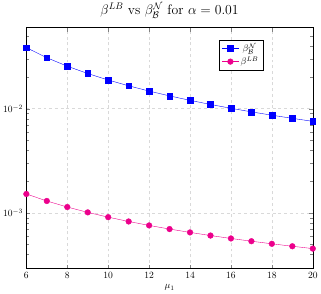}
\caption{\centering }
\label{fig:babu_lb_2_2}
\end{subfigure}
\hfill
\begin{subfigure}[b]{0.24\textwidth}
\centering
\includegraphics[width=\textwidth]{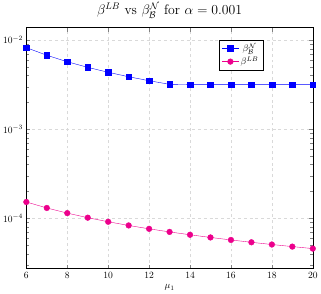}
\caption{\centering }
\label{fig:babu_lb_3_2}
\end{subfigure}
\caption{Comparison of the value of the lower bound $\beta^{LB}(\bmu)$ with respect to the exact Babu\v ska inf-sup constant $\beta_{\Cal B}^{\Cal N}(\bmu)$ for $\alpha = 0.1, 0.01, 0.001$ in (a), (b) and (c), respectively. The analysis has been performed varying the value of $\mu_1$ and fixing $\mu_2 = 2$.}
\label{babu_vs_lb_2_2}
\end{figure}

\subsubsection{Steady Case}
We briefly describe the performance of the lower bound $\beta^{LB}_s(\bmu)$ for the problem \eqref{test_2}. The problem with an initial dimension of $\Cal N = N_h = 10366$ is reduced to $4N = 44$, i.e., the greedy stops at $N=11$ {with a tolerance $\tau = 1\cdot 10^{-4}$}. In Figure \ref{s_geo}, we present some representative solutions for state and adjoint variables while an averaged performance analysis is considered in Table \ref{tab:test_2_s}, where errors \eqref{eq:errors_s} are shown, together with an effectivity and error estimator behaviors\footnote{See footnote \ref{error_estimator}.}. Also in this case, in terms of effectivity, by definition, the Babu\v ska inf-sup constant gives better results, but once again it pays in the offline basis construction, since the computation of the exact value $\beta^{N_h}_{\Cal B_s}(\bmu)$ takes, averagely, $0.8$ seconds. On the contrary. we recall that the computation of the bound only needs 0.09 seconds. As in the time dependent framework, the effectivity is related to the value of the penalization parameter as well as the value of $\mu_2$. We show the plots of the two constants varying $\alpha$ and for $\mu_2 = 1,2$ in Figures \ref{s_babu_vs_lb_geo1} and \ref{s_babu_vs_lb_geo2}, respectively. In terms of computational time needed for a reduced simulation, we reach a speed up of 130, on average.
\begin{figure}[H]
\centering
\begin{subfigure}[b]{0.3\textwidth}
\centering
\includegraphics[width=\textwidth]{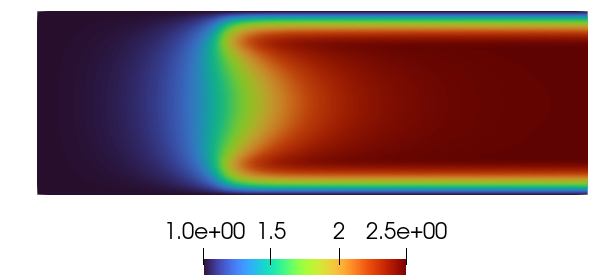}
\caption{\centering HF State Variable}
\label{fig:off_y_geo}
\end{subfigure}
\begin{subfigure}[b]{0.3\textwidth}
\centering
\includegraphics[width=\textwidth]{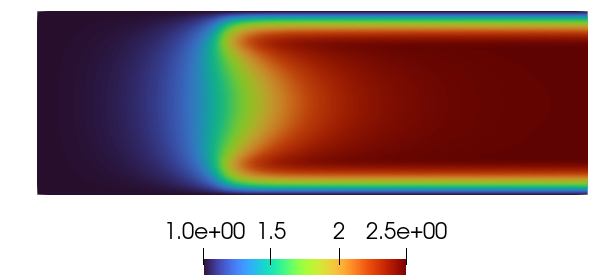}
\caption{\centering RB State Variable}
\label{fig:on_y_geo}
\end{subfigure}\\
\begin{subfigure}[b]{0.3\textwidth}
\centering
\includegraphics[width=\textwidth]{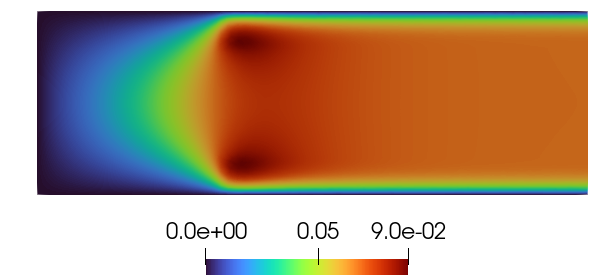}
\caption{\centering HF Adjoint Variable}
\label{fig:off_p_geo}
\end{subfigure}
\begin{subfigure}[b]{0.3\textwidth}
\centering
\includegraphics[width=\textwidth]{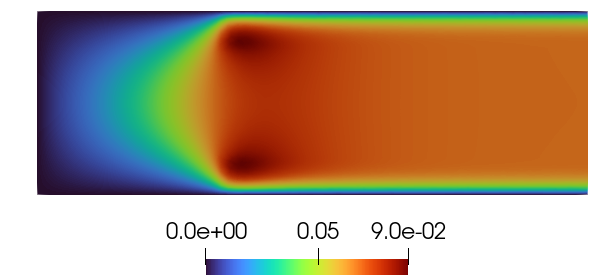}
\caption{\centering RB Adjoint Variable}
\label{fig:on_p_geo}
\end{subfigure}
\hfill
\caption{Optimal HF and RB solutions with $\alpha = 0.07$ and $\bmu=(12.0, 2.0, 2.5)$. HF state variable in (a) and RB state variable in (b). Analogously, HF adjoint variable in (c) and RB adjoint variable in (d).}
\label{s_geo}
\end{figure}
{
\begin{remark}
In comparison with the existing literature on ROMs for time-dependent optimal control, we want to underline that in terms of accuracy and reliability, the estimator is comparable with other approaches and numerical results, as the ones presented in \cite{dede2010reduced,Karcherrev} for distributed time depenendent control problems in an unconstrained framework. The main advantage of the space-time formulation is its versatility and its similarities with the steady problem. This translates in slight modifications with respect to steady experiments, and it facilitates tests and replicability.   
Moreover, we recall that the bound we present is also suited for steady problems, alleviating the computational costs of the offline stage, which does not rely on the approximation of the  Babu\v ska inf-sup constant, improving the results of \cite{negri2013reduced}.
{The space-time all-at-once formulation naturally provides a the saddle point structure that closely resembles the steady optimality system, allowing for a more straightforward analysis building on the steady setting. Moreover, the space-time framework presents several advantages in terms of stability properties, avoiding repeated state/adjoint solves typically required to reach the optimum value. For an overview on the appealability of space-time methods, the reader may refer to \cite{Nochetto2017}, while, for a comparison between space-time constrained optimization and semi-variational methods based on time-marching schemes, we refer the reader to \cite{Beranek2023}.}
\end{remark}
}
\begin{table}[H]
\caption{\centering Steady case: performance analysis for the problem \eqref{test_2}. Average error, estimators and effectivity {$\eta =\Delta_N(\bmu)/\norm{e}_{Y\times Y}$} exploiting the lower bound $\beta_s^{LB}(\bmu)$ and the Babu\v ska inf-sup constant $\beta_s^{N_h}(\bmu)$, with respect to $N$.}
\label{tab:test_2_s}
\footnotesize{\begin{center}
\begin{tabular}{c|c|c|c|c|c|c|c|c}
\hline
\multirow{2}{*}{$N$} & \multicolumn{4}{c|}{$\beta^{LB}_s(\bmu)$}      & \multicolumn{4}{c}{$\beta_{\Cal B_s}^{N_h}(\bmu)$}                         \\ \cline{2-9}
                     & $\norm{e}_{\text{rel}}$ &$\norm{e}_{\text{abs}}$ & $\Delta_N(\bmu)$    & $\eta$            & $\norm{e}_{\text{rel}}$ &$\norm{e}_{\text{abs}}$ & $\Delta_N(\bmu)$    & $\eta$   \\ \hline
1                    & 1.53e--1&2.26e--1     & 1.17e+2   & 1.81e+2 & 2.18e--1 & 3.40e--1 & 3.69e+0 & 1.03e+1 \\ \hline
5                    & 7.34e--3&8.70e--3   & 2.74e+1   & 1.61e+2 & 8.57e--3&7.06e--3 & 7.69e-14&9.12e+0\\ \hline
11                    & 2.13e--4&2.00e--4    & 6.63e-2  & 1.77e+2 & 2.21e--4& 2.16e--4& 2.08e--4&8.34e+0\\ \hline
\end{tabular}\end{center}
}
\end{table}

\begin{figure}[H]
\centering
\begin{subfigure}[b]{0.24\textwidth}
\centering
\includegraphics[width=\textwidth]{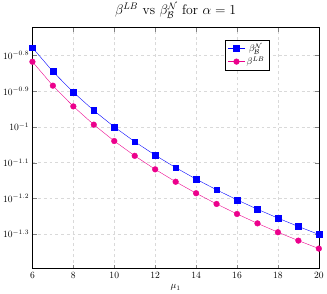}
\caption{\centering }
\label{fig:babu_lb_2_steady_geo}
\end{subfigure}
\hfill
\begin{subfigure}[b]{0.24\textwidth}
\centering
\includegraphics[width=\textwidth]{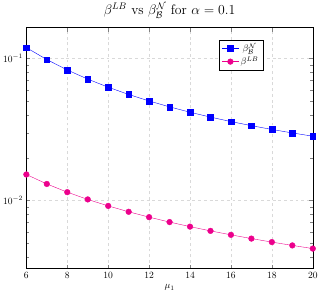}
\caption{\centering }
\label{fig:babu_lb_2_steady_geo}
\end{subfigure}
\hfill
\centering
\begin{subfigure}[b]{0.24\textwidth}
\centering
\includegraphics[width=\textwidth]{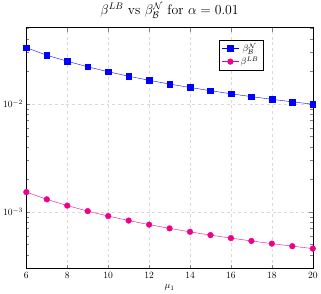}
\caption{\centering }
\label{fig:babu_lb_01_steady_geo}
\end{subfigure}
\hfill
\begin{subfigure}[b]{0.24\textwidth}
\centering
\includegraphics[width=\textwidth]{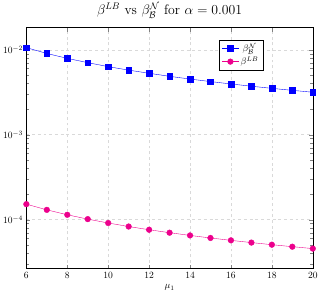}
\caption{\centering }
\label{fig:babu_lb_001_steady_geo}
\end{subfigure}
\caption{Comparison of the value of the lower bound $\beta^{LB}_s(\bmu)$ with respect to the exact Babu\v ska inf-sup constant $\beta_{\Cal B_s}^{N_h}(\bmu)$ for $\alpha = 1, 0.1, 0.01, 0.001$ in (a), (b), (c) and (d), respectively. The analysis has been performed by varying the value of $\mu_1$ and fixing $\mu_2=1$.}
\label{s_babu_vs_lb_geo1}
\end{figure}

\begin{figure}[H]
\centering
\begin{subfigure}[b]{0.24\textwidth}
\centering
\includegraphics[width=\textwidth]{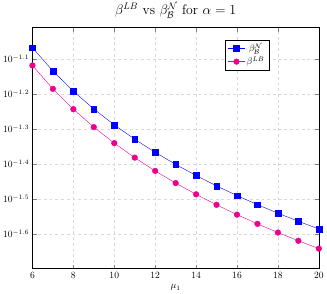}
\caption{\centering }
\label{fig:babu_lb_2_steady_geo2}
\end{subfigure}
\hfill
\begin{subfigure}[b]{0.24\textwidth}
\centering
\includegraphics[width=\textwidth]{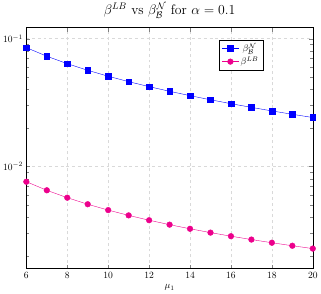}
\caption{\centering }
\label{fig:babu_lb_2_steady_geo2}
\end{subfigure}
\hfill
\centering
\begin{subfigure}[b]{0.24\textwidth}
\centering
\includegraphics[width=\textwidth]{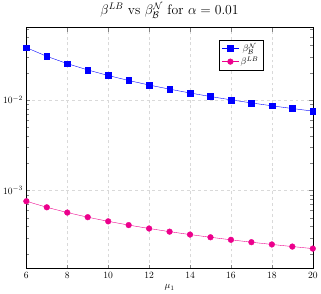}
\caption{\centering }
\label{fig:babu_lb_01_steady_geo2}
\end{subfigure}
\hfill
\begin{subfigure}[b]{0.24\textwidth}
\centering
\includegraphics[width=\textwidth]{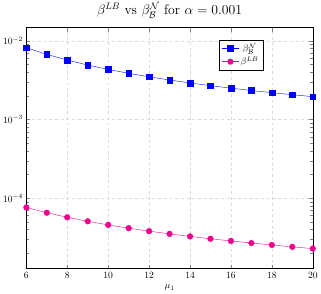}
\caption{\centering }
\label{fig:babu_lb_001_steady_geo2}
\end{subfigure}
\caption{Comparison of the value of the lower bound $\beta^{LB}_s(\bmu)$ with respect to the exact Babu\v ska inf-sup constant $\beta_{\Cal B_s}^{N_h}(\bmu)$ for $\alpha = 1, 0.1, 0.01, 0.001$ in (a), (b), (c) and (d), respectively. The analysis has been performed by varying the value of $\mu_1$ and fixing $\mu_2=2$.}
\label{s_babu_vs_lb_geo2}
\end{figure}

%% file: conclusions.tex
\section{Conclusions}
\label{conclusions}
In this work, we presented a formulation for \ocp s governed by linear parabolic equations. We proposed a well-posedness analysis of the problem at the discrete space-time level. We underline that classical discretization techniques may be too costly to deal with the task of optimization for several parameters. Then, we relied on a greedy algorithm, extending the known literature of steady cases to unsteady ones. The main novelty is given by the proposed new error bound, made by quantities that are known a priori. The strength of the bound is its great versatility since it is valid for unsteady and steady problems. Concerning the latter point, to the best of our knowledge, this is an improvement too, since the literature often relies on very expensive approximation algorithms \cite{negri2013reduced}. At the space-time level, the proposed greedy algorithm not only assures a rapid and reliable online phase but lightens the offline phase of such expensive problem formulation, which, to the best of our knowledge, was performed though standard Proper Orthogonal Decomposition approach \cite{Strazzullo2}. The performance has been tested through distributed \ocp s with physical and geometrical parametrization. We reach high speed up values, due to the high-dimensionality of the all-at-once space-time \ocp s. This is a first step towards the applicability of ROM for \ocp s in a very real-time context, where time evolution optimization is required. \\
Finally, we conclude with some perspectives related to the proposed estimator. First of all, the bound could be sharpened to achieve better results in terms of effectivities. A natural and interesting step could be the extension to more complicated state equations, like Stokes equations, or settings, like {boundary control and $\Omega_u \neq \Omega_{obs}$, for which a deeper analysis is needed}. This will enlarge the applicability of the greedy algorithm in space-time formulation in several fields: biomedical, industrial, environmental, etc. {For example, another extension of interest in applications is related to time-dependent desired states to study phenomena as turnpike  property \cite{dorfman1987linear}, of great importance in econometric.} {A further development can take into account also control constraints and parametric inference as did in \cite{Karcherrev} and \cite{qian2017certified}, respectively. {Another research direction to pursue could concern endtime-observation or control actions of the form $$\sum_{i=1}^{N_u} u_i(t)f_i,$$  with $u_i$ square integrable control amplitudes and $f_i \in Y^*$ given control actions, }. {Moreover, besides the space-time mixed formulation, other techniques and bound may be studied, following the space-time approaches proposed in the high-fidelity context in  \cite{r5,r1,r2,r3,r4}.} 
{Moreover. the application of hybrid POD-Greedy approaches \cite{haasdonk2008reduced} is certainly of interest and might lead to a more efficient reduced strategy.}
We believe that the presented model, enhanced with these improvements, would be versatile in more complicated settings, {featuring, for example, strong transient behaviors}.  Indeed, the main goal is to provide a very general tool to be used in real time contexts for different applications with the purpose of planning and management action.

%% file: BIB.bib
@Article{Babuska1971,
author="Babu{\v{s}}ka, I.",
title="Error-bounds for finite element method",
journal="Numerische Mathematik",
year="1971",
monthf="Jan",
day="01",
volume="16",
number="4",
pages="322--333",
issn="0945-3245",
doi="10.1007/BF02165003",
url="https://doi.org/10.1007/BF02165003"
}

@ARTICLE{PichiStrazzullo,
	author = {Pichi, Federico and Strazzullo, Maria and Ballarin, Francesco and Rozza, Gianluigi},
	title = {Driving bifurcating parametrized nonlinear PDEs by optimal control strategies: Application to Navier-Stokes equations with model order reduction},
	year = {2022},
	journal = {ESAIM: Mathematical Modelling and Numerical Analysis},
	volume = {56},
	number = {4},
	pages = {1361 – 1400},
	doi = {10.1051/m2an/2022044}}

@article{tesi,
  title={Reduced basis method for parametrized optimal control problems governed by {PDE}s},
  author={Negri, F.},
  journal={Master thesis, Politecnico di Milano},
  year={2011},
}

@article{RozzaHuynhManzoni2013,
author = {Rozza, G. and Huynh, D.B.P. and Manzoni, A.},
title = {Reduced basis approximation and a posteriori error estimation for
{S}tokes flows in parametrized geometries: Roles of the inf-sup stability
constants},
journal = {Numerische Mathematik},
year = {2013},
volume = {125},
pages = {115--152},
number = {1},
doi={10.1007/s00211-013-0534-8}
}

@article{prud2002reliable,
  title={Reliable real-time solution of parametrized partial differential equations: Reduced-basis output bound methods},
  author={Prud{'}Homme, C. and Rovas, D. V. and Veroy, K. and Machiels, L. and Maday, Y. and Patera, A. and Turinici, G.},
  journal={Journal of Fluids Engineering},
  volume={124},
  number={1},
  pages={70--80},
  year={2002},
  publisher={American Society of Mechanical Engineers},
  doi={10.1115/1.1448332}
}

@article{dede2010reduced,
  title={Reduced basis method and a posteriori error estimation for parametrized linear-quadratic optimal control problems},
  author={Ded{\`e}, L.},
  journal={SIAM Journal on Scientific Computing},
  volume={32},
  number={2},
  pages={997--1019},
  year={2010},
  publisher={SIAM},
  doi={10.1137/090760453}
}

@Article{Ballarin2017,
author="Ballarin, F.
and Faggiano, E.
and Manzoni, A.
and Quarteroni, A.
and Rozza, G.
and Ippolito, S.
and Antona, C.
and Scrofani, R.",
title="Numerical modeling of hemodynamics scenarios of patient-specific coronary artery bypass grafts",
journal="Biomechanics and Modeling in Mechanobiology",
year="2017",
monthf="Aug",
day="01",
volume="16",
number="4",
pages="1373--1399",
issn="1617-7940",
doi="10.1007/s10237-017-0893-7",
url="https://doi.org/10.1007/s10237-017-0893-7"
}

@ARTICLE{LassilaManzoniQuarteroniRozza2013a,
author = {Lassila, T. and Manzoni, A. and Quarteroni, A. and Rozza, G.},
title = {A reduced computational and geometrical framework for inverse problems
in hemodynamics},
journal = {International Journal for Numerical Methods in Biomedical Engineering},
year = {2013},
volume = {29},
pages = {741--776},
number = {7},
doi = {10.1002/cnm.2559},
}

@article{schoberl2007symmetric,
  title={Symmetric indefinite preconditioners for saddle point problems with applications to {PDE}-constrained optimisation problems},
  author={Sch{\"o}berl, J. and Zulehner, W.},
  journal={SIAM Journal on Matrix Analysis and Applications},
  volume={29},
  number={3},
  pages={752--773},
  year={2007},
  publisher={SIAM}
}

@article{negri2013reduced,
  title={Reduced basis method for parametrized elliptic optimal control problems},
  author={Negri, F. and Rozza, G. and Manzoni, A. and Quarteroni, A.},
  journal={SIAM Journal on Scientific Computing},
  volume={35},
  number={5},
  pages={A2316--A2340},
  year={2013},
  publisher={SIAM},
  doi={10.1137/120894737}
}

@article{negri2015reduced,
  title={Reduced basis approximation of parametrized optimal flow control problems for the {S}tokes equations},
  author={Negri, F. and Manzoni, A. and Rozza, G.},
  journal={Computers \& Mathematics with Applications},
  volume={69},
  number={4},
  pages={319--336},
  year={2015},
  publisher={Elsevier},
  doi={10.1016/j.camwa.2014.12.010}
}

@article{kunisch2008proper,
  title={Proper orthogonal decomposition for optimality systems},
  author={Kunisch, K. and Volkwein, S.},
  journal={ESAIM: Mathematical Modelling and Numerical Analysis},
  volume={42},
  number={1},
  pages={1--23},
  year={2008},
  publisher={EDP Sciences},
  doi={10.1051/m2an:2007054}
}

@article{karcher2014certified,
  title={A certified reduced basis method for parametrized elliptic optimal control problems},
  author={K{\"a}rcher, M. and Grepl, M. A.},
  journal={ESAIM: Control, Optimisation and Calculus of Variations},
  volume={20},
  number={2},
  pages={416--441},
  year={2014},
  publisher={EDP Sciences},
  doi={10.1007/s10915-017-0539-z}
}

@article{karcher2018certified,
  title={Certified reduced basis methods for parametrized elliptic optimal control problems with distributed controls},
  author={K{\"a}rcher, M. and Tokoutsi, Z. and Grepl, M. A. and Veroy, K.},
  journal={Journal of Scientific Computing},
  volume={75},
  number={1},
  pages={276--307},
  year={2018},
  publisher={Springer},
  doi={10.1007/s10915-017-0539-z}
}

@article{bader2015certified,
  title={A Certified Reduced Basis Approach for Parametrized Linear-Quadratic Optimal Control Problems with Control Constraints},
  author={Bader, E. and K{\"a}rcher, M. and Grepl, M. A. and Veroy-Grepl, K.},
  journal={IFAC-PapersOnLine},
  volume={48},
  number={1},
  pages={719--720},
  year={2015},
  publisher={Elsevier},
  doi = {10.1016/j.ifacol.2015.05.167},
}

@article{bader2016certified,
  title={Certified reduced basis methods for parametrized distributed elliptic optimal control problems with control constraints},
  author={Bader, E. and K{\"a}rcher, M. and Grepl, M. A. and Veroy, K.},
  journal={SIAM Journal on Scientific Computing},
  volume={38},
  number={6},
  pages={A3921--A3946},
  year={2016},
  publisher={SIAM},
  doi="10.1137/16M1059898"
}

@book{hinze2008optimization,
  title={Optimization with {PDE} constraints},
  author={Hinze, M.L. and Pinnau, R. and Ulbrich, M. and Ulbrich, S.},
  volume={23},
  year={Antwerp, 2008},
  publisher={Springer Science \& Business Media}
}

@book{leugering2014trends,
  title={Trends in {PDE} constrained optimization},
  author={Leugering, G. and Benner, P. and Engell, S. and Griewank, A. and Harbrecht, H. and Hinze, M. and Rannacher, R. and Ulbrich, S.},
  year={New York, 2014},
  publisher={Springer}
}

@article{HinzeStokes,
  title={A Hierarchical Space-Time Solver for Distributed Control of the {S}tokes Equation},
  author={Hinze, M.L. and K\"oster, M. and Turek, S.},
  journal={Technical Report, SPP1253-16-01}, 
  year={2008}
}

@incollection{HinzeNS,
  title={A space-time multigrid method for optimal flow control},
  author={Hinze, M.L. and K{\"o}ster, M. and Turek, S.},
  booktitle={Constrained optimization and optimal control for partial differential equations},
  pages={147--170},
  year={2012},
  publisher={Springer}, 
  doi={10.1007/978-3-0348-0133-1_8}
}

@book{boffi2013mixed,
  title={Mixed finite element methods and applications},
  author={Boffi, D. and Brezzi, F. and Fortin, M.},
  volume={44},
  year={Berlin and Heidelberg, 2013},
  publisher={Springer-Verlag}
}

@book{makinen,
  title={Introduction to shape optimization: theory, approximation, and computation},
  author={Haslinger, J. and M{\"a}kinen, R. A. E.},
  year={ Philadelphia, 2003},
  publisher={SIAM}
}

@book{delfour2011shapes,
  title={Shapes and geometries: metrics, analysis, differential calculus, and optimization},
  author={Delfour, M. C. and Zol{\'e}sio, J.},
  volume={22},
  year={Philadelphia, 2011},
  publisher={SIAM}
}

@book{quarteroni2008numerical,
  title={Numerical approximation of partial differential equations},
  author={Quarteroni, A. and Valli, A.},
  volume={23},
  year={Berlin and Heidelberg, 2008},
  publisher={Springer Science \& Business Media}
}

@article{dede2007optimal,
  title={Optimal flow control for {N}avier-{S}tokes equations: Drag minimization},
  author={Ded\`e, L.},
  journal={International Journal for Numerical Methods in Fluids},
  volume={55},
  number={4},
  pages={347--366},
  year={2007},
  publisher={Chichester; New York: Wiley, c1981-},
  doi={10.1002/fld.1464}
}

@article{optimal,
  title={Optimal control of {N}avier--{S}tokes equations by {O}seen approximation},
  author={Po{\v{s}}ta, M. and Roub{\'\i}{\v{c}}ek, T.},
  journal={Computers \& Mathematics With Applications},
  volume={53},
  number={3},
  pages={569--581},
  year={2007},
  publisher={Elsevier},
  doi={10.1016/j.camwa.2006.02.034}
}

@article{de2007optimal,
  title={Optimal control of the stationary {N}avier-{S}tokes equations with mixed control-state constraints},
  author={de los Reyes, J. C. and Tr{\"o}ltzsch, F.},
  journal={SIAM Journal on Control and Optimization},
  volume={46},
  number={2},
  pages={604--629},
  year={2007},
  publisher={SIAM},
  doi={10.1137/050646949}
}

@inproceedings{quarteroni2007reduced,
  title={Reduced basis methods for optimal control of advection-diffusion problems},
  author={Quarteroni, A. and Rozza, G. and Quaini, A.},
  booktitle={Advances in Numerical Mathematics},
  number={CMCS-CONF-2006-003},
  pages={193--216},
  year={2007},
  organization={RAS and University of Houston}
}

@book{mohammadi2010applied,
  title={Applied shape optimization for fluids},
  author={Mohammadi, B. and Pironneau, O.},
  year={New York, 2010},
  publisher={Oxford University Press}
}

@article{RozzaHuynhPatera2008,
author = {Rozza, G. and Huynh, D.B.P. and Patera, A.T.},
title = {Reduced basis approximation and a posteriori error estimation for
affinely parametrized elliptic coercive partial differential equations:
Application to transport and continuum mechanics},
journal = {Archives of Computational Methods in Engineering},
year = {2008},
volume = {15},
pages = {229--275},
number = {3},
doi = {10.1007/s11831-008-9019-9}
}

@article{hesthaven2015certified,
  title={Certified reduced basis methods for parametrized partial differential equations},
  author={Hesthaven, J. S. and Rozza, G. and Stamm, B.},
  journal={SpringerBriefs in Mathematics},
  year={2015, Springer, Milano},
  publisher={Springer}
}

@BOOK{Quarteroni20151,
	author = {Quarteroni, Alfio and Manzoni, Andrea and Negri, Federico},
	title = {Reduced basis methods for partial differential equations: An introduction},
	year = {2015},
	journal = {Reduced Basis Methods for Partial Differential Equations: An Introduction},
	pages = {1 – 263},
	doi = {10.1007/978-3-319-15431-2},
	publication_stage = {Final}
}

@article{gerner2012certified,
  title={Certified reduced basis methods for parametrized saddle point problems},
  author={Gerner, A. L. and Veroy, K.},
  journal={SIAM Journal on Scientific Computing},
  volume={34},
  number={5},
  pages={A2812--A2836},
  year={2012},
  publisher={SIAM}, 
  doi={10.1137/110854084}
}

@book{fenics,
  title={Automated Solution of Differential Equations by the Finite Element Method},
  author={Logg, A. and Mardal, K.A. and Wells, G.},
  year={Berlin, 2012},
  publisher={Springer-Verlag}
}

@book{rbnics,
  author    = {Rozza, Gianluigi and Ballarin, Francesco and Scandurra, Leonardo and Pichi, Federico},
  title     = {Real Time Reduced Order Computational Mechanics},
  subtitle  = {Parametric PDEs Worked Out Problems},
  publisher = {Springer Cham},
  year      = {2024},
  series    = {SISSA Springer Series},
  isbn      = {978-3-031-49891-6}
}

@article{Strazzullo1,
author = {Strazzullo, M. and Ballarin, F. and Mosetti, R. and Rozza, G.},
title = {Model Reduction for Parametrized Optimal Control Problems in Environmental Marine Sciences and Engineering},
journal = {SIAM Journal on Scientific Computing},
volume = {40},
number = {4},
pages = {B1055-B1079},
year = {2018},
preprint = {https://arxiv.org/abs/1710.01640},
doi = {10.1137/17M1150591}
}

@article{Stoll,
 author = {Stoll, M. and Wathen, A.},
 title = {All-at-once Solution of Time-dependent {S}tokes Control},
 journal = {J. Comput. Phys.},
 issue_date = {January, 2013},
 volume = {232},
 number = {1},
 monthf = jan,
 year = {2013},
 issn = {0021-9991},
 pages = {498--515},
 numpages = {18},
 urlf = {http://dx.doi.org/10.1016/j.jcp.2012.08.039},
 doi = {10.1016/j.jcp.2012.08.039},
 acmid = {2397451},
 publisher = {Academic Press Professional, Inc.},
 address = {San Diego, CA, USA},
}

@article{Stoll1,
author = {Stoll, M. and Wathen, A.},
year = {Technical report. 2010},
pages = {},
journal = {},
title = {All-at-once solution of time-dependent {PDE}-constrained optimization problems}
}

@article{Benzi, 
title={Numerical solution of saddle point problems},
volume={14},
DOI={10.1017/S0962492904000212},
journal={Acta Numerica},
publisher={Cambridge University Press},
author={Benzi, M. and Golub, G. H. and Liesen, J.}, year={2005}, pages={1–137}}

@article{yano2014space,
  title={A space-time {P}etrov--{G}alerkin certified reduced basis method: Application to the {B}oussinesq equations},
  author={Yano, M.},
  journal={SIAM Journal on Scientific Computing},
  volume={36},
  number={1},
  pages={A232--A266},
  year={2014},
  publisher={SIAM},
  doi={10.1137/120903300}
}

@article{yano2014,
  title={A space-time hp-interpolation-based certified reduced basis method for {B}urgers' equation},
  author={Yano, M. and Patera, A. T. and Urban, K.},
  journal={Mathematical Models and Methods in Applied Sciences},
  volume={24},
  number={09},
  pages={1903--1935},
  year={2014},
  publisher={World Scientific},
  doi={10.1142/S0218202514500110}
}

@Inbook{Glas2017,
author="Glas, Silke
and Mayerhofer, Antonia
and Urban, Karsten",
editor="Benner, Peter
and Ohlberger, Mario
and Patera, Anthony
and Rozza, Gianluigi
and Urban, Karsten",
title="Two Ways to Treat Time in Reduced Basis Methods",
bookTitle="Model Reduction of Parametrized Systems",
year="2017",
publisher="In Model Reduction of Parametrized Systems. Springer International Publishing",
address="Cham",
pages="1--16",
isbn="978-3-319-58786-8",
doi="10.1007/978-3-319-58786-8_1",
url="https://doi.org/10.1007/978-3-319-58786-8_1"
}

@article{urban2012new,
  title={A new error bound for reduced basis approximation of parabolic partial differential equations},
  author={Urban, K. and Patera, A. T.},
  journal={Comptes Rendus Mathematique},
  volume={350},
  number={3-4},
  pages={203--207},
  year={2012},
  publisher={Elsevier},
  doi={10.1016/j.crma.2012.01.026}
}

@article{eriksson1987error,
  title={Error estimates and automatic time step control for nonlinear parabolic problems, {I}},
  author={Eriksson, K. and Johnson, C.},
  journal={SIAM Journal on Numerical Analysis},
  volume={24},
  number={1},
  pages={12--23},
  year={1987},
  publisher={SIAM},
  doi={10.1137/0724002}
}

@article{seymen2014distributed,
  title={Distributed optimal control of time-dependent diffusion--convection--reaction equations using space--time discretization},
  author={Seymen, Z. K. and Y{\"u}cel, H. and Karas{\"o}zen, B.},
  journal={Journal of Computational and Applied Mathematics},
  volume={261},
  pages={146--157},
  year={2014},
  publisher={Elsevier},
  doi={10.1016/j.cam.2013.11.006}
}

@InProceedings{ZakiaMaria,
  author    = {Strazzullo, Maria and Zainib, Zakia and Ballarin, Francesco and Rozza, Gianluigi},
  title     = {Reduced Order Methods for Parametrized Non-linear and Time Dependent Optimal Flow Control Problems, Towards Applications in Biomedical and Environmental Sciences},
  booktitle = {Numerical Mathematics and Advanced Applications ENUMATH 2019},
  editor    = {Vermolen, F. J. and Vuik, C.},
  series    = {Lecture Notes in Computational Science and Engineering},
  volume     = {139},
  pages      = {841--850},
  year       = {2021},
  publisher  = {Springer},
  address    = {Cham},
  doi        = {10.1007/978-3-030-55874-1_83}
}

@article{CARERE2021261,
title = {A weighted {POD}-reduction approach for parametrized {PDE}-constrained optimal control problems with random inputs and applications to environmental sciences},
journal = {Computers \& Mathematics with Applications},
volume = {102},
pages = {261-276},
year = {2021},
issn = {0898-1221},
doi = {https://doi.org/10.1016/j.camwa.2021.10.020},
author = {Giuseppe Carere and Maria Strazzullo and Francesco Ballarin and Gianluigi Rozza and Rob Stevenson}
}

@article{Zoccolan1,
urlf = {https://doi.org/10.1515/cmam-2023-0171},
title = {A Streamline Upwind {P}etrov-{G}alerkin Reduced Order Method for Advection-Dominated Partial Differential Equations Under Optimal Control},
title = {},
author = {Fabio Zoccolan and Maria Strazzullo and Gianluigi Rozza},
journal = {Computational Methods in Applied Mathematics},
doi = {10.1515/cmam-2023-0171},
year = {2024}
}

@article{zoccolan2,
urlf = {https://doi.org/10.1515/jnma-2023-0006},
title = {Stabilized weighted reduced order methods for parametrized advection-dominated optimal control problems governed by partial differential equations with random inputs},
title = {},
author = {Fabio Zoccolan and Maria Strazzullo and Gianluigi Rozza},
journal = {Journal of Numerical Mathematics},
doi = {10.1515/jnma-2023-0006},
year = {2024}
}

@article{Strazzullo2,
author = {M. Strazzullo and F. Ballarin and G. Rozza},
title = {{POD--Galerkin Model Order Reduction for Parametrized Time Dependent Linear Quadratic Optimal Control Problems in Saddle Point Formulation}},
journal = {Journal of Scientific Computing},
volume = {83},
number = {3},
pages = {55},
year = {2020},
preprint = {https://arxiv.org/abs/1909.09631},
doi = {10.1007/s10915-020-01232-x}
}

@incollection{StrazzulloZuazua,
title = {Chapter 9 - {S}pace-time {POD}-{G}alerkin approach for parametric flow control},
editor = {Emmanuel Tr\`elat and Enrique Zuazua},
series = {Handbook of Numerical Analysis},
publisher = {Elsevier},
volume = {23},
pages = {307-338},
year = {2022},
booktitle = {Numerical Control: Part A},
issn = {1570-8659},
doi = {10.1016/bs.hna.2021.12.009},
author = {Francesco Ballarin and Gianluigi Rozza and Maria Strazzullo}
}

@article{buffa2012priori,
  title={A priori convergence of the greedy algorithm for the parametrized reduced basis method},
  author={Buffa, A. and Maday, Y. and Patera, A. T. and Prud’homme, C. and Turinici, G.},
  journal={ESAIM: Mathematical modelling and numerical analysis},
  volume={46},
  number={3},
  pages={595--603},
  year={2012},
  publisher={EDP Sciences},
  doi={10.1051/m2an/2011056}
}

@article{necas,
  title={Les m{\'e}thodes directes en th{\'e}orie des {\'e}quations elliptiques},
  author={Ne{\v{c}}as, J.},
  year={1967},
  publisher={Masson}
}

@article{langer2020,
  title={Unstructured space-time finite element methods for optimal control of parabolic equations},
  author={Langer, U. and Steinbach, O. and Tr{\"o}ltzsch, F. and Yang, H.},
  journal={SIAM Journal on Scientific Computing},
  volume={43},
  number={2},
  pages={A744--A771},
  year={2021},
  publisher={SIAM},
  doi={10.1137/20M1330452}
}

@article{xu2003some,
  title={Some observations on {B}abu{\v{s}}ka and {B}rezzi theories},
  author={Xu, Jinchao and Zikatanov, Ludmil},
  journal={Numerische Mathematik},
  volume={94},
  number={1},
  pages={195--202},
  year={2003},
  publisher={Springer},
  doi={10.1007/s002110100308}
}

@article{Strazzullo3,
title={{POD}-{G}alerkin Model Order Reduction for Parametrized
Nonlinear Time Dependent Optimal Flow Control: an
Application to {S}hallow {W}ater {E}quations},
author={Strazzullo, M. and Ballarin, F. and Rozza, G.},
doi = {10.1515/jnma-2020-0098},
journal = {Journal of Numerical Mathematics},
number = {1},
volume = {30},
year = {2022},
pages = {63--84},
lastchecked = {2022-04-29}
}

@article{Iapichino1,
author = {Iapichino, L. and Ulbrich, S. and Volkwein, S.},
title = {Multiobjective {PDE}-Constrained Optimization Using the Reduced-Basis Method},
year = {2017},
issue_date = {October 2017},
publisher = {Springer-Verlag},
address = {Berlin, Heidelberg},
volume = {43},
number = {5},
issn = {1019-7168},
urlf = {10.1007/s10444-016-9512-x},
doi = {10.1007/s10444-016-9512-x},
journal = {Adv. Comput. Math.},
monthf = oct,
pages = {945–972},
numpages = {28}
}

@InProceedings{Iapichino2,
author="Iapichino, L.
and Trenz, S.
and Volkwein, S.",
editor="Karas{\"o}zen, B{\"u}lent
and Manguo{\u{g}}lu, Murat
and Tezer-Sezgin, M{\"u}nevver
and G{\"o}ktepe, Serdar
and U{\u{g}}ur, {\"O}m{\"u}r",
title="Reduced-Order Multiobjective Optimal Control of Semilinear Parabolic Problems",
booktitle="Numerical Mathematics and Advanced Applications  ENUMATH 2015",
year="2016",
publisher="Springer International Publishing",
address="Cham",
pages="389--397",
isbn="978-3-319-39929-4"
}

@article{haasdonk2008reduced,
  title={Reduced basis method for finite volume approximations of parametrized linear evolution equations},
  author={Haasdonk, B. and Ohlberger, M.},
  journal={ESAIM: Mathematical Modelling and Numerical Analysis-Mod{\'e}lisation Math{\'e}matique et Analyse Num{\'e}rique},
  volume={42},
  number={2},
  pages={277--302},
  year={2008},
  doi={10.1051/m2an:2008001}
}

@article{huynh2010natural,
  title={A natural-norm successive constraint method for inf-sup lower bounds},
  author={Huynh, DBP and Knezevic, DJ and Chen, Y and Hesthaven, Jan S and Patera, AT},
  journal={Computer Methods in Applied Mechanics and Engineering},
  volume={199},
  number={29-32},
  pages={1963--1975},
  year={2010},
  publisher={Elsevier},
  doi={10.1016/j.cma.2010.02.011}
}

@ARTICLE{Karcherrev,
author={K{\"a}rcher, M. and Grepl, M. A.},
title={A posteriori error estimation for reduced order solutions of parametrized parabolic optimal control problems},
journal={ESAIM: Mathematical Modelling and Numerical Analysis},
year={2014},
volume={48},
number={6},
pages={1615-1638},
doi={10.1051/m2an/2014012}
}

@article{qian2017certified,
  title={A certified trust region reduced basis approach to PDE-constrained optimization},
  author={Qian, Elizabeth and Grepl, Martin and Veroy, Karen and Willcox, Karen},
  journal={SIAM Journal on Scientific Computing},
  volume={39},
  number={5},
  pages={S434--S460},
  year={2017},
  publisher={SIAM}
}

@article{r1,
author = {Meidner, Dominik and Vexler, Boris},
title = {{A Priori Error Estimates for Space-Time Finite Element Discretization of Parabolic Optimal Control Problems Part {I}: Problems Without Control Constraints}},
journal = {SIAM Journal on Control and Optimization},
volume = {47},
number = {3},
pages = {1150-1177},
year = {2008},
doi = {10.1137/070694016},

URL = { 
        https://doi.org/10.1137/070694016
    
}
}

@article{r2,
author = {Meidner, Dominik and Vexler, Boris},
title = {{A Priori Error Estimates for Space-Time Finite Element Discretization of Parabolic Optimal Control Problems Part II: Problems with Control Constraints}},
journal = {SIAM Journal on Control and Optimization},
volume = {47},
number = {3},
pages = {1301-1329},
year = {2008},
doi = {10.1137/070694028},

URL = { 
        https://doi.org/10.1137/070694028
    
}
}

@article{r3,
author = {Meidner, Dominik and Vexler, Boris},
title = {A Priori Error Analysis of the Petrov–Galerkin Crank–Nicolson Scheme for Parabolic Optimal Control Problems},
journal = {SIAM Journal on Control and Optimization},
volume = {49},
number = {5},
pages = {2183-2211},
year = {2011},
doi = {10.1137/100809611},

URL = { 
        https://doi.org/10.1137/100809611
    
}
}

@ARTICLE{r4,
author={Springer, A. and Vexler, B.},
title={Third order convergent time discretization for parabolic optimal control problems with control constraints},
journal={Computational Optimization and Applications},
year={2014},
volume={57},
number={1},
pages={205-240},
doi={10.1007/s10589-013-9580-5}}

@ARTICLE{Hinze200545,
	author = {Hinze, M.},
	title = {A variational discretization concept in control constrained optimization: The linear-quadratic case},
	year = {2005},
	volume = {30},
	number = {1},
	pages = {45 – 61},
	doi = {10.1007/s10589-005-4559-5}
	}

@ARTICLE{r5,
author={Von Daniels, N.},
title={Tikhonov regularization of control-constrained optimal control problems},
journal={Computational Optimization and Applications},
year={2018},
volume={70},
number={1},
pages={295-320},
doi={10.1007/s10589-017-9976-8},
note={cited By 4},
document_type={Article},
source={Scopus}
}

@book{dorfman1987linear,
  title={Linear programming and economic analysis},
  author={Dorfman, Robert and Samuelson, Paul Anthony and Solow, Robert M},
  year={1987},
  publisher={Courier Corporation}
}

@InProceedings{adj,
author="Akman, Tu{\u{g}}ba
and Karas{\"o}zen, B{\"u}lent",
editor="Carraro, Thomas
and Geiger, Michael
and K{\"o}rkel, Stefan
and Rannacher, Rolf",
title="Space-Time Discontinuous Galerkin Methods for Optimal Control Problems Governed by Time Dependent Diffusion-Convection-Reaction Equations",
booktitle="Multiple Shooting and Time Domain Decomposition Methods",
year="2015",
publisher="Springer International Publishing",
address="Cham",
pages="233--261",
isbn="978-3-319-23321-5", 
doi={10.1007/978-3-319-23321-5_9}
}

@article{eriksson1985time2,
  title={Time discretization of parabolic problems by the discontinuous {G}alerkin method},
  author={Eriksson, Kenneth and Johnson, Claes and Thom{\'e}e, Vidar},
  journal={ESAIM: Mathematical Modelling and Numerical Analysis},
  volume={19},
  number={4},
  pages={611--643},
  year={1985},
  publisher={EDP Sciences},
  doif={M2AN_1985__19_4_611_0}
}

@article{STRAZZULLO2023128191,
title = {{POD-Based reduced order methods for optimal control problems governed by parametric partial differential equation with varying boundary control}},
journal = {Applied Mathematics and Computation},
volume = {457},
pages = {128191},
year = {2023},
issn = {0096-3003},
doi = {10.1016/j.amc.2023.128191},
urlf = {https://www.sciencedirect.com/science/article/pii/S0096300323003600},
author = {Maria Strazzullo and Fabio Vicini}
}

@article{MALKUS19811299,
title = {Eigenproblems associated with the discrete {LBB} condition for incompressible finite elements},
journal = {International Journal of Engineering Science},
volume = {19},
number = {10},
pages = {1299-1310},
year = {1981},
issn = {0020-7225},
doi = {10.1016/0020-7225(81)90013-6},
urlf = {https://www.sciencedirect.com/science/article/pii/0020722581900136},
author = {D.S. Malkus}
}

@article{HinzeNew,
author = {Ali, Ahmad Ahmad and Hinze, Michael},
title = {Reduced Basis Methods---An Application to Variational Discretization of Parametrized Elliptic Optimal Control Problems},
journal = {SIAM Journal on Scientific Computing},
volume = {42},
number = {1},
pages = {A271-A291},
year = {2020},
doi = {10.1137/18M1227147}}

@article{Hendrik,
  author  = {Hendrik Kleikamp and Martin Lazar and Cesare Molinari},
  title   = {Be greedy and learn: efficient and certified algorithms for parametrized optimal control problems},
  journal = {ESAIM: Mathematical Modelling and Numerical Analysis},
  volume   = {59},
  number   = {1},
  pages    = {291--330},
  year     = {2025},
  doi      = {10.1051/m2an/2024074}
}

@incollection{Hinzelibro,
  author    = {Carmen Gr{\"a}{\ss}le and Michael Hinze and Stefan Volkwein},
  title     = {Model Order Reduction by Proper Orthogonal Decomposition},
  booktitle = {Model Order Reduction: Volume 2 -- Snapshot-Based Methods and Algorithms},
  editor    = {Peter Benner and Stefano Grivet-Talocia and
               Alfio Quarteroni and Gianluigi Rozza and
               Wil Schilders and Luis Miguel Silveira},
  publisher = {De Gruyter},
  pages      = {47--96},
  year       = {2021},
  doi        = {10.1515/9783110671490-002}
}

@article{ManzoniFast,
  author  = {Carlo Sinigaglia and Davide E. Quadrelli and
             Andrea Manzoni and Francesco Braghin},
  title   = {Fast active thermal cloaking through PDE-constrained optimization and reduced-order modelling},
  journal = {Proceedings of the Royal Society A:
             Mathematical, Physical and Engineering Sciences},
  volume  = {478},
  number  = {2258},
  pages   = {20210813},
  year    = {2022},
  doi     = {10.1098/rspa.2021.0813}}

@article{Beranek2023,
  author    = {Nina Beranek and Martin Alexander Reinhold and Karsten Urban},
  title     = {A Space--Time Variational Method for Optimal Control Problems: Well-Posedness, Stability and Numerical Solution},
  journal   = {Computational Optimization and Applications},
  year      = {2023},
  volume     = {86},
  number     = {2},
  pages      = {767--794},
  doi        = {10.1007/s10589-023-00507-x}
}

@article{Nochetto2017,
  author  = {Ricardo H. Nochetto and Stefan A. Sauter and Christian Wieners},
  title   = {Space-time Methods for Time-dependent Partial Differential Equations},
  journal = {Oberwolfach Reports},
  year    = {2017},
  volume  = {14},
  number  = {1},
  pages   = {863--947},
  doi     = {10.4171/OWR/2017/15}
}
